\documentclass[a4paper,12pt]{amsart}
\usepackage{amscd,amsthm,amsfonts,latexsym,amssymb}

\newtheorem{thm}{Theorem}[section]
\newtheorem{lem}[thm]{Lemma}
\newtheorem{cor}[thm]{Corollary}
\newtheorem{prop}[thm]{Proposition}

\theoremstyle{definition}

\newtheorem{defn}[thm]{Definition}
\newtheorem{rem}[thm]{Remark}
\newtheorem{exm}[thm]{Example}

\numberwithin{equation}{section}

\begin{document}

\renewcommand{\vec}[1]{\boldsymbol{#1}}

\renewcommand{\log}{\ln} 

\newcommand{\dd}{{\mathrm d}}  
\newcommand{\pa}{{\partial}}
 
\newcommand{\ZZ}{{\mathbb Z}}  
\newcommand{\RR}{{\mathbb R}}  
\newcommand{\CC}{{\mathbb C}}  
\newcommand{\Ci}{{\mathbb C} \cup \{ \infty\} }  
\newcommand{\HH}{{\mathbb H}}  

\newcommand{\Ll}{{\mathcal L}}  
\newcommand{\Hh}{{\mathcal H}}  
\newcommand{\Vv}{{\mathcal V}}  
\newcommand{\Mm}{{\mathcal M}}  
\newcommand{\Aa}{{\mathcal A}}  

\newcommand{\CP}{{\mathbb C}P} 
\newcommand{\HP}{{\mathbb H}P} 
\newcommand{\Ff}{{\mathcal F}}  

\newcommand{\ii}{{\rm i}}  
\newcommand{\jj}{{\rm j}}  
\newcommand{\eu}{{\rm e}}  

\newcommand{\iI}{{\mathbf i}}  
\newcommand{\jJ}{{\mathbf j}}  
\newcommand{\rR}{{\mathbf r}}  
\newcommand{\kK}{{\mathbf k}}  
\newcommand{\lL}{{\mathbf l}}  
\newcommand{\sS}{{\mathbf s}}  

\newcommand{\SO }{\mbox{\rm SO}} 
\newcommand{\U}{\mbox{\rm U}}   
\newcommand{\Jj}{{\mathcal J}}  
\newcommand{\JJ}{{\mathrm J}}  
\newcommand{\Tt}{{\mathcal T}}  
\newcommand{\SU}{\mbox{\rm SU}}  
\newcommand{\PU}{\mbox{\rm PU}}  
\newcommand{\Sp}{\mbox{\rm Sp}}  
\newcommand{\SL}{\mbox{\rm SL}}  

\newcommand{\ra}{\rightarrow}
\newcommand{\ds}{\displaystyle}

\newcommand{\inn}[1]{\langle#1\rangle} 
\newcommand{\biginn}[1]{\big\langle#1\big\rangle} 
\newcommand{\Biginn}[1]{\Big\langle#1\Big\rangle} 
\newcommand{\bigginn}[1]{\bigg\langle#1\bigg\rangle} 
\newcommand{\Bigginn}[1]{\Bigg\langle#1\Bigg\rangle} 

\newcommand{\norm}[1]{\lvert#1\rvert} 
\newcommand{\bignorm}[1]{\big\lvert#1\big\rvert} 
\newcommand{\Bignorm}[1]{\Big\lvert#1\Big\rvert} 
 
\newcommand{\hh}{\text{\rm Herm}}  
\newcommand{\cc}{c} 

\newcommand{\Harm}{\text{\rm Harm}}  
\newcommand{\Hol}{\text{\rm Hol}}  
\newcommand{\HHol}{\text{\rm HHol}}  
\newcommand{\ff}{\text{\rm full}}  
\newcommand{\nf}{\text{\rm non}} 
\newcommand{\hor}{\text{\rm hor}} 

\newcommand{\End}{\text{\rm End}}  

\newcommand{\Ver}{\text{\rm Vert}} 


\renewcommand{\phi}{\varphi}
\newcommand{\la}{\lambda}
\newcommand{\La}{\Lambda}
\newcommand{\na}{\nabla}
\newcommand{\al}{\alpha}
\newcommand{\be}{\beta}
\newcommand{\ga}{\gamma}
\newcommand{\Ga}{\Gamma}
\newcommand{\ep}{\epsilon}
\newcommand{\ve}{\varepsilon}
\newcommand{\si}{\sigma}
\newcommand{\Si}{\Sigma}
\newcommand{\ze}{\zeta}
\newcommand{\om}{\omega}
\newcommand{\Om}{\Omega}

\newcommand{\wt}{\widetilde} 
\newcommand{\ov}{\overline} 
\newcommand{\wh}{\widehat} 
\newcommand{\un}{\underline} 

\renewcommand{\Re}{{\rm Re}\,} 
\renewcommand{\Im}{{\rm Im}\,} 
\renewcommand{\div}{\mbox{\rm div}\,} 

\newcommand{\grad}{\mbox{\rm grad}\,}

\newcommand{\nablacirc}{\overset{\circ}{\na}} 

\hyphenation{Eu-clid-ean}

\newcommand{\Tr}{\mbox{\rm Trace}} 
\newcommand{\Id}{\mbox{\rm Id}} 
\newcommand{\spn}{\mbox{\rm span}} 
\newcommand{\image}{\mbox{\rm Image}}  
\newcommand{\codim}{\mbox{\rm codim}}  

\newcommand{\volnostyle}{\bf}

\title[Jacobi fields along harmonic 2-spheres in $S^3$ and $S^4$.]{Jacobi fields along harmonic 2-spheres in $S^3$ and $S^4$ are not all integrable. }

\author{Luc Lemaire}
\author{John C. Wood.}

\address{D\'epartement de Math\'ematique, Universit\'e libre de Bruxelles, C.P. 218,
Bd.\ du Triomphe, 1050 Bruxelles, Belgium}
\address{Department of Pure Mathematics, University of Leeds\\
Leeds LS2 9JT, Great Britain}
\email{llemaire@ulb.ac.be; j.c.wood@leeds.ac.uk}

\begin{abstract}In a previous paper, we showed that any Jacobi field along a harmonic map from the 2-sphere to the complex projective plane is integrable (i.e., is tangent to a
smooth variation through harmonic maps).  In this paper, in contrast, we show that there are (non-full) harmonic maps from the 2-sphere to the 3-sphere and 4-sphere which have non-integrable Jacobi fields.   This is particularly surprising in the case of the 3-sphere where the space of harmonic maps of any degree is a smooth manifold, each map having image in a totally geodesic 2-sphere.
\end{abstract}

\subjclass[2000]{Primary 58E20, Secondary 53C43}
\keywords{harmonic map; Jacobi field; infinitesimal deformation.}

\maketitle

\thispagestyle{empty}

\section*{Introduction}
A \emph{harmonic map} between smooth Riemannian manifolds is a smooth map which extremizes the energy functional --- a natural generalization of the Dirichlet integral.
A \emph{Jacobi field} along a harmonic map is an infinitesimal deformation of the harmonic map which preserves harmonicity `to first order'.  

It is important to know whether a Jacobi field is \emph{integrable}, i.e., tangent to a family of maps which are \emph{genuinely} harmonic --- not just to first order. 
We shall study this question for maps between compact manifolds.  Then, if all the Jacobi fields along the harmonic maps between two given real-analytic Riemannian manifolds are integrable, it follows that the space of harmonic maps between those manifolds is a finite-dim\-en\-sion\-al real-analytic manifold with tangent spaces given by the Jacobi fields \cite{AdSi}.  The case that we study here of \emph{maps from the $2$-sphere} to a Riemannian manifold $N$ has special interest; for instance, the integrability of Jacobi fields is related to the properties of the singular set of weakly harmonic maps \emph{from an arbitrary} Riemannian manifold to $N$ \cite{Sim}  (see \cite[p.\ 471]{LeWo2}).

Whether the Jacobi fields are integrable is known only for a handful of cases, see \cite{LeWo2}.   In the present paper, we consider the case of $2$-sphere to higher-dim\-en\-sion\-al spheres.  Recall that a smooth map from the $2$-sphere is harmonic if and only if it is a \emph{minimal branched immersion} in the sense of \cite{GuOsRo}.   The construction of all harmonic $2$-spheres in spheres  has been understood for a long time, see \cite{Ca1,Ca2,Ch1,Ch2,Ba, Br}, however, the only case where the integrability question was settled was for harmonic maps from the $2$-sphere to itself \cite{GuWh}.

Thinking of the codomain $2$-sphere as $\CP^1$, it is natural to consider next maps into the complex projective plane
$\CP^2$.  In this case, the authors \cite{LeWo2} showed that \emph{all Jacobi fields along harmonic $2$-spheres in
$\CP^2$ are integrable}, by analysing the construction \cite{EeWo} of all harmonic maps.

In contrast, in the present paper we show that, \emph{along some harmonic $2$-spheres in $S^3$ or $S^4$, there are non-integrable Jacobi fields}.  {}From the twistor theory for harmonic maps, we see that the space of harmonic maps $\phi:S^2 \to S^4$ of fixed twistor degree is an algebraic variety  \cite{Ve2}; at a non-smooth point $\phi$, there are non-integrable Jacobi fields.  When the twistor degree is at least three, the algebraic variety is the union of three components with transverse intersection; we call the maps in the intersections of those components \emph{collapse points}.  Such maps are not full, but are the limits of families of full maps; they are non-smooth points, and so have non-integrable Jacobi fields.  

The case of maps into $S^3$ follows from the $S^4$ case, but is more surprising since the space of harmonic $2$-spheres in $S^3$ of any fixed twistor degree is a smooth manifold.   At some points of this manifold, there are extra Jacobi fields not in its tangent space, and so non-integrable.

A main tool is to show that \emph{Jacobi fields along a full harmonic map $\phi:S^2 \to S^4$ can be lifted to deformations of its twistor lift} --- a non-trivial result in the presence of branch points --- this allow us to examine the Jacobi fields by lifting the problem to the holomorphic category.  For non-full maps, we may not be able to lift Jacobi fields; instead, we use a correspondence between Jacobi fields and eigenfunctions of a Schr\"odinger equation \cite{MoRo,EjKo-extra,Ej-min, Ej-bound,Ko}. 

The paper is arranged as follows.  In \S \ref{sec:prelims}, we recall the twistor construction of harmonic maps from $S^2$ to $S^4$ in terms of horizontal holomorphic maps from $S^2$ to $\CP^3$, and study the infinitesimal deformations of such holomorphic maps. 

In \S \ref{sec:main}, we show how the twistor construction gives diffeomorphisms between the spaces of holomorphic and harmonic maps (Theorem \ref{th:smooth}).  We then show that Jacobi fields can be lifted, thus giving bijections between the infinitesimal deformations of the harmonic and holomorphic data (Theorem \ref{th:Jacobi-lift}). 
We use this in \S \ref{sec:nonfull} to show the non-integrability of some Jacobi fields along certain non-full harmonic maps into $S^4$.   Then using the Schr\"odinger equation as above, we deduce the same property in the case of $S^3$.

For a harmonic map from a $2$-sphere, equivalently a minimal branched immersion, the (real) dimension of the space of Jacobi fields along the map is called its \emph{nullity (for the energy)}.  In \S 4, we relate this to the \emph{nullity for the area}, and compare our results with those of S. Montiel and F. Urbano for minimal \emph{immersions} of $S^2$ in $S^4$.

In \S \ref{sec:collapse}, we study the subvariety of collapse points, equivalently, the subvariety of maps admitting an extra eigenfunction, showing that it is non-empty for all twistor degrees $d \geq 3$, and finding it explicitly for the important case $3 \leq d \leq 5$.
We discuss the influence of branch points: we note first that, when $d=3$, collapse points occur precisely when the four branch points exhibit the most symmetry; we then analyse the case $d=4$ with branch points already exhibiting some symmetry, and find the further necessary and sufficient conditions to give collapse points.

For general facts on Jacobi fields along harmonic maps, see \cite{Wo-Rome, Sim}; some of the results of this paper were announced in \cite{Wo-Tenerife}.

The second author thanks Luis Fern\'andez, Bruno Ascenso Sim\~oes and Martin Svensson for lively discussions on this work, and the Belgian FNRS for financial support for part of this work.  This work is part of a programme started by a question posed by L.\ Simon during the first MSJ International Research Institute in Tohoku University, Sendai (1993).


\section{Preliminaries}\label{sec:prelims}

\subsection{Harmonic maps} \label{subsec:ha-map}

Harmonic maps are defined to be the solutions to a variational problem as follows (see, for example, \cite{EeSam,EeLe-Rpt1, EeLe-Sel,EeLe-Rpt2}).

Let $M = (M^m,g)$ and $N = (N^n,h)$ be compact smooth Riemannian manifolds without boundary of arbitrary (finite) dimensions\/ $m$ and $n$ respectively, and let  $\phi:(M,g) \to (N,h)$  be a smooth map between them.  Define the \emph{energy} of $\phi$ by
\begin{equation} \label{energy}
E(\phi) = \frac{1}{2} \int_M |\dd\phi|^2 \,\om_g
\end{equation}
where $\om_g$ is the volume  measure on $M$ defined by the metric $g$, and $|\dd\phi|$ is the Hilbert--Schmidt norm of $\dd\phi$ given at each point $x \in M$ by
\begin{equation}
|\dd\phi_x|^2 = \sum_{i=1}^m \inn{\dd\phi_x(e_i), \dd\phi_x(e_i)}
\end{equation}
for any orthonormal basis $\{e_i\}$ of $T_xM$.  Here, and throughout the paper, $\langle \cdot, \cdot \rangle$ denotes the inner product on the relevant bundle $E \to M$ induced from the metrics on $M$ and $N$ and $| \cdot |$ the corresponding norm given by $|v| = \sqrt{\langle v, v \rangle}$ \ $(v \in E)$.  Further, we shall use $\Ga(E)$ to denote the space of smooth sections of $E$.

By a \emph{smooth (one-parameter) variation $\Phi = \{\phi_t\}$ of $\phi$} we mean a smooth map
$\Phi:M \times (-\epsilon, \epsilon) \to N$, \ $\Phi(x,t) = \phi_t(x)$, where $\epsilon >0$ and
$\phi_0 = \phi$.   The \emph{variation vector field of\/ $\Phi = \{\phi_t\}$} is defined by $v = \pa\phi_t\!/\pa t \vert_{t=0}$\,, this is a \emph{vector field along $\phi$}, i.e., a section of the pull-back bundle $\phi^{-1}TN \to M$; we shall say that $v$ is \emph{tangent to the variation $\{\phi_t\}$}.

A smooth map $\phi:(M,g) \to (N,h)$ is called \emph{harmonic} if it is an \emph{extremal} of the energy integral,
 i.e., for all smooth one-parameter variations $\{\phi_t\}$ of $\phi$, the \emph{first variation} $\frac{\dd}{\dd t} E(\phi_t)\big\vert_{t=0}$ is zero.   We compute (see, for example, \cite{EeLe-Sel}):
\begin{equation} \label{1st-var}
\frac{\dd}{\dd t} E(\phi_t)\Big\vert_{t=0} = - \int_M \big\langle \tau(\phi), v \big\rangle \, \om_g \,,
\end{equation}
where $\tau(\phi) \in \Ga(\phi^{-1}TN)$ is the \emph{tension field of\/ $\phi$} defined by
\begin{multline*}
\tau(\phi) = \Tr\,\na \dd\phi =
	\sum_{i=1}^m \na \dd\phi(e_i, e_i) \\[-2ex]
	= \sum_{i=1}^m \bigl\{ \na^{\phi}_{e_i}\! \bigl(\dd\phi(e_i)\bigr) - \dd\phi(\na^M_{e_i}e_i) \bigr\} \,.
\end{multline*}
Here $\na^M$ denotes the Levi-Civita connection on $M$, $\na^{\phi}$ the pull-back of the Levi-Civita connection $\na^N$ on $N$ to the bundle $\phi^{-1}TN \to M$, and $\na$ the tensor product connection on the bundle $T^*M \otimes \phi^{-1}TN$ induced from these connections. 
Equation \eqref{1st-var} says that $\tau(\phi)$ is the negative of the gradient at $\phi$ of the energy functional $E$ on a suitable space of mappings, i.e., it points in the direction in which $E$ decreases most rapidly \cite[(3.5)]{EeLe-Rpt1}.  It follows from \eqref{1st-var} that $\phi$ is harmonic if and only if it satisfies the \emph{harmonicity equation}: $\tau(\phi) = 0$.

For any manifold $M$, $T^{\cc}M = TM \otimes \CC$ will denote the \emph{complexified tangent bundle}.  When $(M,J)$ is an almost complex manifold, this decomposes into the direct sum of the holomorphic (or $(1,0)$-) tangent bundle $T'M$ and the antiholomorphic (or $(0,1)$-) tangent bundle $T''M$; these being the $(+\ii)$- and $(-\ii)$-eigenspaces of $J$.   When $(M,J)$ is a complex manifold, these three bundles are holomorphic bundles, and, as is standard, we shall use the map $v \mapsto$ its $(1,0)$-part $v' = \frac{1}{2}(v - \ii Jv)$ to identify $TM$ with $T'M$, thus giving $TM$ a holomorphic structure under which the action of $J$ on $TM$ corresponds to multiplication by $\ii$ in $T'M$.

Now suppose that $(M^2,g)$ is a $2$-dim\-en\-sion\-al Riemannian manifold.  Then the energy \eqref{energy}, and so  harmonicity, of a map depend only on the conformal structure induced by $g$.  In fact, let $(x_1,x_2)$ be isothermal coordinates and write $z = x_1+ \ii x_2$\,.  Write $\pa/\pa z = \frac{1}{2} (\pa/\pa x_1 - \ii \pa/ \pa x_2)$ and $\pa/\pa\ov{z} = \frac{1}{2} (\pa/\pa x_1 + \ii \pa/ \pa x_2)$; the harmonicity equation can then be written (see, for example \cite[\S 3.5]{BaWo-book}):
\begin{equation} \label{harmonicity}
\na^{\phi}_{\frac{\pa}{\pa \bar z}}\,\frac{\pa \phi}{\pa z} = 0 \quad \text{or, equivalently,} \quad
		\na^{\phi}_{\frac{\pa}{\pa z}}\,\frac{\pa \phi}{\pa \bar z} = 0 \,.
\end{equation}

When $M^2$ is oriented, by taking charts consisting of \emph{oriented} isothermal coordinates $(x_1,x_2)$, we give $M^2$ the structure of a one-dim\-en\-sion\-al complex manifold, or \emph{Riemann surface} with complex coordinates $z = x_1 + \ii x_2$\,.  Then $\pa/\pa z$ and $\pa/\pa\ov{z}$ provide local bases for $T'M$ and $T''M$, respectively.   

For $n \in \{1,2,\ldots,\}$, let $S^n$ be the unit sphere in $\RR^{n+1}$ with the induced metric, and let $\CP^n$ be the $n$-dimensional complex projective space with its Fubini--Study metric of constant holomorphic sectional curvature.  We shall frequently identify the complex projective line $\CP^1$ with the $2$-sphere by the mapping
\begin{equation} \label{CP1S2}
[z_0,z_1] \mapsto
\frac{1}{|z_0|^2 + |z_1|^2} \bigl(|z_0|^2 - |z_1|^2, \, \ov{z_0} z_1 \bigr) \,;
\end{equation}
this is biholomorphic and an isometry up to scale; we further identify $\CP^1$ conformally (in fact, biholomorphically) with the extended complex plane $\Ci$ by the mapping
$[z_0,z_1] \mapsto z_0^{-1}z_1$.  The composition of these two identifications is the biholomorphic map given by \emph{stereographic projection} $\si:S^2 \to \Ci$:
\begin{equation} \label{stereo}
z = \si(x_1,x_2,x_3) = (x_2 + \ii x_3)\big/(1+x_1) \,.
\end{equation}

A map $\phi:M^2 \to N$ is called \emph{weakly conformal} if, away from points where $\dd\phi$ is zero, it preserves angles; in a local complex coordinate $z$ on $M^2$, this can be expressed by the equation
\begin{equation} \label{WC}
\Biginn{\frac{\pa\phi}{\pa z}, \frac{\pa\phi}{\pa z}}^{\!\!\cc} = 0 \,.
\end{equation} 
Here $\inn{\cdot, \cdot}^{\!\cc}$ denotes the complex-bilinear extension of the inner product $\inn{\cdot, \cdot}$ on $N$.  Note that \eqref{WC} says that $\pa\phi/\pa z$ is orthogonal to $\pa\phi/\pa\bar{z}$ with respect to the \emph{Hermitian extension} of the inner product:
\begin{equation} \label{Herm-extn}
\biginn{v, w}^{\!\hh} = \biginn{v, \ov{w}}^{\!\cc} \qquad (v, w \in T^{\cc}_xN, \ x \in N) \,.
\end{equation}

Note that the energy of a weakly conformal map is equal to its \emph{area}.  As is well-known, any harmonic map from the $2$-sphere is weakly conformal, indeed, the harmonic equation \eqref{harmonicity} shows that the quantity $\biginn{\frac{\pa\phi}{\pa z}, \frac{\pa\phi}{\pa z}}^{\!\!\cc}\dd z^2$ is a well-defined holomorphic differential on $S^2$, and so must vanish.

A \emph{minimal branched immersion} is a smooth map from a Riemann surface which is a conformal minimal immersion except at isolated points where it has branch points in the sense of \cite{GuOsRo}.
A non-constant weakly conformal map is harmonic if and only if it is a minimal branched immersion; in particular, \emph{a smooth map from $S^2$ is harmonic if and only if it is a minimal branched immersion}.

More generally, a smooth map $\phi:M^2 \to N$ is called \emph{real isotropic} if, in any local complex coordinate $z$, the quantities
$$
\eta_{r,s}(\phi) = \Biginn{ \bigl(\na^{\phi}_{\frac{\pa}{\pa z}}\bigr)^{r-1}  \Bigl(\frac{\pa\phi}{\pa z}\Bigr) \:,\, 
	\bigl(\na^{\phi}_{\frac{\pa}{\pa z}}\bigr)^{s-1}  \Bigl(\frac{\pa\phi}{\pa z}\Bigr) }^{\!\cc}
$$
are zero for all integers $r,s \geq 1$. 
Note that this condition is independent of the complex coordinate chosen, and putting $r=s=1$ shows that any real-isotropic map is weakly conformal.

Let $\phi:S^2 \to S^n$ be a harmonic map from the $2$-sphere to an $n$-sphere ($n \geq 2$). Then \cite{Ca1,Ca2}
$\phi$ is real isotropic.  Indeed, one shows by induction on $k=r+s$ that
$\eta_{r,s}(\phi)\dd z^k$ defines a \emph{holomorphic $k$-differential}, i.e., a holomorphic section of $\otimes^kT'_*S^2$, and so must vanish. 

Let $i:S^n \to \RR^{n+1}$ denote the standard isometric inclusion mapping.   For a smooth map $\phi:M^2 \to S^n$, write
\begin{equation} \label{varPhi}
\varPhi = i \circ \phi \,.
\end{equation}

Then (see, for example, \cite{EeLe-Rpt1, BaWo-book}), $\phi$ is harmonic if and only if

\begin{equation} \label{ha-sphere}
\frac{\pa^2\varPhi}{\pa\ov{z}\,\pa z} = - \Bignorm{\frac{\pa\varPhi}{{\pa z}}}^2\,\varPhi
\end{equation}
or, more invariantly,
\begin{equation} \label{ha-sphere0}
\Delta^M\varPhi = \norm{\dd\varPhi}^2 \varPhi
\end{equation}
where $\Delta^M$ is the Laplace-Beltrami operator on $(M,g)$ (conventions as in \cite{EeLe-Sel, Wo-Rome}).

The real isotropy of $\phi$ can be expressed in terms of $\varPhi$ by
\begin{equation} \label{isotropy}
\Biginn{\frac{\pa^r\varPhi}{\pa z^r}, \frac{\pa^s\varPhi}{\pa z^s}}^{\!\cc} = 0 \quad
		\text{for all } r,s \in \{0,1,2,\ldots\} \text{ with } r+s \geq 1.
\end{equation}
Note that \eqref{isotropy} holds automatically for $r+s = 1$, and coincides with the weak conformality condition \eqref{WC} for $(r,s) = (1,1)$.

The real isotropy allows us to construct all harmonic maps from $S^2$ to $S^n$ explicitly
from holomorphic data, see \cite{Ca1,Ca2,Ch1,Ch2,Ba,Br}; the case $n=4$ is rather special, we shall recall that construction below.

\subsection{Infinitesimal deformations of harmonic maps}

Let $\phi:(M,g) \to (N,h)$ be harmonic.  We can describe the
\emph{second variation} of the energy at $\phi$ as
follows.  Let $v, w \in \Ga(\phi^{-1}TN)$.  Choose a (smooth) two-parameter variation $\Phi =\{\phi_{t,s}\}$ of $\phi$  with
$$
\left.\frac{\pa\phi_{t,s}}{\pa t}\right\vert_{(t,s)=(0,0)} = v \quad \text{and} \quad
\left.\frac{\pa\phi_{t,s}}{\pa s}\right\vert_{(t,s)=(0,0)} = w \,.
$$
The \emph{Hessian of\/ $\phi$} is defined by 
\begin{equation} \label{Hessian1}
H_{\phi}(v,w) =
	\left.\frac{\pa^2 E(\phi_{t,s})}{\pa t\pa s} \right\vert_{(t,s)=(0,0)}.
\end{equation}  
This depends only on $v$ and $w$; indeed, it is given by the \emph{second variation formula} (see, for example, \cite{EeLe-Sel}):
\begin{equation} \label{Hessian2}
H_{\phi}(v,w) = \int_M \inn{\JJ_{\phi}(v), w} \,\om_g
\end{equation}
where
$$
\JJ_{\phi}(v) = \Delta^{\phi} v - \Tr\, R^N\!(\dd\phi, v)\,\dd\phi \,;
$$
here $\Delta^{\phi}$ denotes the Laplacian on $\phi^{-1}TN$ and $R^N$ the curvature
operator of $N$ (conventions as in \cite{EeLe-Sel,LeWo2,Wo-Tokyo,Wo-Rome}).
The mapping $\JJ_{\phi}: \Ga(\phi^{-1}TN) \to \Ga(\phi^{-1}TN)$ is called the
\emph{Jacobi operator} (for the energy); it is a self-adjoint
linear elliptic operator.   A vector field $v$ along $\phi$ is called a \emph{Jacobi field (along
$\phi$)} if it is in the kernel of the Jacobi operator, i.e., it
satisfies the Jacobi equation
\begin{equation} \label{Jacobi}
\JJ_{\phi}(v) = 0 \,.
\end{equation}
By standard elliptic theory, the set of Jacobi fields along a
harmonic map is a finite-dim\-en\-sion\-al vector subspace of $\Ga(\phi^{-1}TN)$.

We shall make use of the following interpretation of the Jacobi operator
as the \emph{linearization} of the tension field \cite{LeWo2}.

\begin{defn} Let
$\{\phi_t\}$ be a smooth 1-parameter family of maps from $(M,g)$ to $(N,h)$.
Say that\/ $\{\phi_t\}$ is \emph{harmonic to first
order} if its tension field is \emph{zero to first order} in the sense that (see the next footnote for the meaning of the derivative $\pa/\pa t$)  
\begin{equation} \label{first-order}
\tau(\phi_0) = 0 \quad \text{and} \quad \left. \frac{\pa}{\pa t}\tau(\phi_t)\right\vert_{t=0} = 0 \,.
\end{equation}
\end{defn}

We shall write the condition \eqref{first-order} succinctly as
\begin{equation} \label{order-t}
\tau(\phi_t) = o(t)\,.
\end{equation}

\begin{prop} \label{prop:first-order}
Let\/ $\phi:M \to N$ be harmonic and let\/ $v \in \Ga(\phi^{-1}TN)$. 
Let
$\Phi = \{\phi_t\}$ be a smooth variation of\/ $\phi$ tangent to $v$.  Then {\rm (}\footnote{Here, and in \cite{LeWo2,Wo-Tokyo}, this means that the components of each side with respect to a local frame on $N$ satisfy $\JJ_{\phi}(v)^{\al} = - (\pa/\pa t)\tau(\phi_t)^{\al}\vert_{t=0}$\,.  Alternatively $\pa/\pa t$ in \eqref{Jacobi-interp} and \eqref{first-order} can be replaced by the covariant derivative $\na^{\Phi}_{\pa/\pa t}$ in the pull-back bundle $\Phi^{-1}TN$.  The resulting expressions are all equal since $\tau(\phi) = 0$.}{\rm )}
\begin{equation} \label{Jacobi-interp}
\JJ_{\phi}(v) = - \left. \frac{\pa}{\pa t}\tau(\phi_t) \right\vert_{t=0} \,.
\end{equation}

In particular, $v$ is a Jacobi field along $\phi$ if and only if\/ $\{\phi_t\}$ is harmonic to first order: $\tau(\phi_t) = o(t)$.
\qed \end{prop}

In particular, if $\{\phi_t\}$ is a smooth variation of $\phi$ through harmonic
maps, its variation vector field
$\ds v = \pa\phi_t\big/\pa t\vert_{t=0}$ is a Jacobi field.
This suggests the following definition.
\begin{defn} \label{def:integrable}
A Jacobi field $v$ along a harmonic map $\phi:M \to N$ is called \emph{integrable} if it is tangent to a
smooth variation $\{\phi_t\}$ of $\phi$ through harmonic maps, i.e., there exists a one-parameter family $\{\phi_t\}$ of harmonic maps with $\phi_0 = \phi$ and $\pa\phi_t\big/\pa t\vert_{t=0} = v$.
\end{defn}

When all Jacobi fields are integrable, many consequences follow, see \cite{LeWo2}; here we quote only one.

\begin{prop}\label{prop:AdamsSimon} \cite{AdSi}
Let $\phi_0:(M,g) \to (N,h)$ be a harmonic map between
real-analytic manifolds.   Then all Jacobi fields along
$\varphi_0$ are integrable if and only if the space of harmonic maps
($C^{2,\alpha}$-)close to $\varphi_0$ is a
smooth manifold whose tangent space at $\varphi_0$ is exactly the space\/
$\ker \JJ_{\varphi_0}$ of Jacobi fields along it.
\qed \end{prop}

A vector field $v$ along a conformal map $\phi:M^2 \to N$ from a surface is called \emph{conformal} if \eqref{WC} is satisfied to first order for any one-parameter variation of $\phi$ tangent to $v$.   Conformality of harmonic maps from a $2$-sphere is preserved to first order by Jacobi fields, see \cite[\S 3.1]{Wo-Tokyo}.  Further, for harmonic maps into spheres, \emph{isotropy} \eqref{isotropy} is preserved to first order, as in the following result, which is equivalent to \cite[Proposition 3.4]{Wo-Tokyo}.

\begin{prop} \label{prop:isotropy-1st}
Let $\phi:S^2 \to S^n$ be a harmonic map, and let $v$ be a Jacobi field along it.  Then $v$ preserves isotropy to first order in the sense that, if\/ $\{\phi_t\}$ is any one-parameter variation of\/ $\phi$ tangent to $v$, then, writing\/ $\varPhi_t =  i \circ \phi_t$ where $i:S^n \to \RR^{n+1}$ is the standard inclusion, we have
\begin{equation} \label{isotropy-1st}
\Biginn{\frac{\pa^r\varPhi_t}{\pa z^r}, \frac{\pa^s\varPhi_t}{\pa z^s}}^{\!\!\cc} = o(t)
\end{equation}
for all integers $r,s \geq 0$ with $r+s \geq 1$.
\qed \end{prop}

\subsection{The twistor space of the $4$-sphere} \label{subsec:twistor-S4}

We recall the well-known construction of the twistor space of the $4$-sphere, and its identification with $\CP^3$ and $\SO(5)/\U(2)$.  The seminal article on this is \cite{AtHiSi}; our account is based on that in \cite[Chapter 7]{BaWo-book}.

Let $M = (M^{2m},g)$ be an oriented Riemannian manifold of even dimension $2m$.  Let $\SO(M) \to M$ denote the bundle whose fibre at $x \in M$ is the set $\SO(T_xM)$ of orthonormal oriented frames at $x$.  Note that $\SO(2m)$ and its subgroup $\U(m)$ act on this set.   Let $x \in M$.  By an \emph{almost complex structure at $x$} (or \emph{on} $T_xM$) we mean a linear
transformation $J_x: T_x M \to T_x M$ such that ${J_x}^2 = -\Id_{T_xM}$; if it is isometric, we call it an \emph{almost Hermitian structure at $x$}.  Given an orthonormal basis $\{e_1, \ldots, e_{2m}\}$ of $T_x M$, setting 
\begin{equation} \label{Jx}
J_x e_{2j-1} = e_{2j}\,, \quad J_x e_{2j} = -e_{2j-1} \quad (j=1, \ldots, m)
\end{equation} 
defines an almost Hermitian structure $J_x$ at $x$ which we call \emph{positive} (respectively, \emph{negative})
according as $\{e_1, \ldots, e_{2m}\}$ is positively (respectively, negatively) oriented.  This defines a map from the set $\SO(T_x M)$ of positively oriented orthonormal frames at $x$ to the set $\Si^+_x = \Si^+_x(M)$ of positive almost Hermitian structures at $x$; that map factors to a bijection $\SO(T_x M)/\U(m) \to \Si^+_x$ which endows $\Si^+_x$ with the structure of a Hermitian symmetric space.
Further, on fixing an orthonormal basis for $T_x M$, we have an isomorphism $\SO(T_x M) \cong \SO(2m)$ which induces an isomorphism of $\Si^+_x$ with $\SO(2m)/\U(m)$\,.

Let $(V, \langle \, \cdot \, \rangle)$ be an inner product space such as (i) $\RR^n$ with its standard inner product or (ii) $T_xM$ with the inner product given by the metric.  Call a subspace $P$ of $V^{\cc} = V \otimes \CC$ \emph{isotropic} if $\inn{v,w}^c = 0$ for all $v,w \in P$.
A subspace $P$ of $T_x^{\cc}M$ is isotropic if and only if it is orthogonal to its complex conjugate $\ov{P}$, with respect to the Hermitian inner product \eqref{Herm-extn}.  Then there is a one-to-one correspondence between the set of all Hermitian structures $J_x$ at $x$ and the set of all $m$-dimensional isotropic subspaces $P$ given by setting $P$ equal to the $(0,1)$-tangent space given by the $(-\ii)$-eigenspace of $J_x$;   we call $P$  \emph{positive}  (respectively, \emph{negative}) according as $J_x$ is positive (respectively, negative).  More explicitly, given an orthonormal basis $\{e_1, \ldots, e_{2m}\}$ of $T_x M$, if $J_x$ is given by \eqref{Jx}, then $P$ is the complex subspace of $T^{\cc}_xM$ spanned by $\{e_1 + \ii e_2\,,\, e_3 + \ii e_4\,,\, \ldots,\, e_{2m-1} + \ii e_{2m}\}$.


Let $\SO(M) \to M$ be the principal bundle of positive orthonormal frames with fibre $\SO(T_xM)$ at $x \in M$.
Then the \emph{(positive) twistor bundle of\/ $M$} is the associated fibre bundle
\begin{equation} \label{twistor-proj}
\begin{array}{rcl}\pi:\Si^+ = \Si^+(M) &= & \SO(M) \times_{\mathrm{SO}(2m)} \SO(2m)/\U(m) \\ 
		&\cong & \SO(M)/\U(m) \longrightarrow M \,. \end{array}
\end{equation}
The fibre of $\pi$ at $x$ is the set $\Si^+_x$ of positive almost Hermitian structures at $x$.
The map $\pi$ is called the \emph{twistor map} or \emph{twistor projection} and its total space $\Si^+$ is called the \emph{(positive) twistor space of\/ $(M,g)$}.
The manifold $\Si^+$ has a canonical almost complex structure obtained as follows.  First, each fibre $\Si^+_x$ of $\pi$ has a complex structure $\Jj^{\Vv}$; indeed, it has the structure of a Hermitian symmetric space as described above.  Call the bundle of tangents to the fibres the
\emph{vertical subbundle} $\Vv(\Si^+)$; then the Levi-Civita connection $\na^M$ of $(M,g)$ defines a complementary subbundle $\Hh(\Si^+)$ of $T\Si^+$, called the
\emph{horizontal subbundle}.  Thus, we have a decomposition
\begin{equation}  \label{HV-Si}
T\Si^+ = \Vv(\Si^+) \oplus \Hh(\Si^+) \,;
\end{equation}
we shall denote the associated projections by the same letters, viz., $\Vv:T\Si^+ \to \Vv(\Si^+)$
and $\Hh:T\Si^+ \to \Hh(\Si^+)$.  Each $w \in \Si^+$ defines an almost complex structure on $T_{\pi(w)}M$; we use the isomorphism defined by the differential $\dd\pi_w|_{\Hh_w(\Si^+)}:\Hh_w(\Si^+) \to T_{\pi(w)}M$ to lift this to an almost complex structure $\Jj^{\Hh}_w$ on $\Hh_w(\Si^+)$.  Then the formula
\begin{equation} \label{twistor:J1}
\Jj_w = (\Jj^{\Vv}_w \,,\,  \Jj^{\Hh}_w) \qquad (w \in \Si^+)
\end{equation}
defines an almost complex structure $\Jj$ on $\Si^+$.  

Later, we shall need to complexify the decomposition \eqref{HV-Si} to a decomposition
\begin{equation} \label{HV-Si-C}
T^{\cc}\Si^+ = \Vv^{\cc}(\Si^+) \oplus \Hh^{\cc}(\Si^+) \,;
\end{equation}
we continue to denote the associated projections by $\Vv$ and $\Hh$.  This last decomposition restricts to a decomposition
\begin{equation} \label{HV-Si'}
T'\Si^+ = \Vv'(\Si^+) \oplus \Hh'(\Si^+)
\end{equation} 
 of the $(1,0)$-tangent bundle; we denote the associated projections by $\Vv'$ and $\Hh'$.
 
We now identify the twistor space $\Si^+(S^4)$ of $S^4$ and corresponding twistor projection
\begin{equation} \label{twistor-S4}
\pi:\Si^+(S^4) \to S^4.
\end{equation}
We use the following general theory for the twistor space of an oriented $4$-dim\-en\-sion\-al Riemannian manifold $(M^4,g)$.  Let $\pi:Z^6 \to M^4$ be a Riemannian submersion from a K\"ahler manifold $(Z^6,G,J)$.  Suppose that $\pi$ has totally geodesic fibres which are connected compact complex submanifolds.  Then we have a direct sum decomposition of bundles over $Z^6$:
\begin{equation} \label{Z6:orthog}
TZ^6 = \Vv(Z^6) \oplus \Hh(Z^6) \,,
\end{equation}
where $\Vv(Z^6)$ is the bundle of tangents to the fibres of $\pi$ and $\Hh(Z^6)$ is its orthogonal complement with respect to the metric $G$.    Then, for each $w \in Z^6$,  $J_w$ restricts to an endomorphism of $\Hh_w(Z^6)$; we use the isomorphism $\dd\pi_w|_{\Hh_w(Z^6)}$ to transfer this to an almost Hermitian structure $\iota(w)$ on $T_{\pi(w)}M^4$, thus defining a bundle map $\iota:Z^6 \to \Si^+(M^4)$.

The \emph{integrability tensor of\/ $\Hh$} is defined by
\begin{equation} \label{int-tensor}
I(X,Y) = \Vv\bigl([X,Y]\bigr) \qquad  (X,Y \in \Ga(\Hh(Z^6))\,) \,.
\end{equation}
Say that $\Hh(Z^6)$ is \emph{nowhere integrable} if this is non-zero at all points. 
Then (see, for example, \cite[\S 7.2]{BaWo-book}), 

\begin{prop} \label{prop:Kahler-twistor-sp}
{\rm (i)} The map $\iota:(Z^6,J) \to (\Si^+(M^4),\Jj)$ is holomorphic and maps $\Hh(Z^6)$ to $\Hh(\Si^+)$.

{\rm (ii)} If\/ $\Hh(Z^6)$ is nowhere integrable, then $\iota$ is a bundle isomorphism.
\qed \end{prop}

Thus, $\pi:Z^6 \to M^4$ provides a model for the twistor bundle $\Si^+$ of $M^4$.
When $M^4$ is the $4$-sphere, there are two realizations of this, as follows.

\smallskip

1) \emph{A quadric Grassmannian as twistor space.} 
For any positive integers $m$ and $n$ with $2m < n$,  we define the \emph{quadric Grassmannian} $\Tt_{m,n}$ to be the following submanifold of the complex Grassmannian $G_m(\CC^n)$:
\begin{equation} \label{quadric-Grass}
 		\begin{array}{rl} \Tt_{m,n} =&\{ P \in G_m(\CC^n) : P \text{ is isotropic} \} \\
	  		=&\{ P \in G_m(\CC^n) : \inn{v,w}^{\cc} = 0 \text{ for all } v,w \in P \} \,. \end{array}  
\end{equation}
The K\"ahler structure on $G_m(\CC^n)$ restricts to a K\"ahler structure on  $\Tt_{m,n}$. 
Clearly, $\Tt_{m,n}$ can be idenfied with the homogeneous space $\SO(n)/\U(m) \times \SO(n-2m)$.   
Since the frame bundle of $S^4$ can be identified with the homogeneous principal bundle
$\SO(5) \to \SO(5)/\SO(4) = S^4$, it follows from \eqref{twistor-proj} that the twistor bundle of $S^4$ is the quadric Grassmannian $\Tt_{2,5} = \SO(5)/\U(2)$ with twistor projection the Riemannian submersion $\SO(5)/\U(2) \to \SO(5)/\SO(4)$ induced by the canonical inclusion of $\U(2)$ in $\SO(4)$.
More geometrically, the twistor projection $\pi:\Tt_{2,5} \to S^4$ is given by
\begin{equation} \label{twistor-proj-T}
\pi(P)  =(P \oplus \ov P)^{\perp}
\end{equation}
where the direct sum is regarded as an oriented real subspace and $(P \oplus \ov P)^{\perp}$ denotes its positively oriented unit normal.   Explicitly, if $\{e_0,e_1,e_2,e_3,e_4\}$ is a positive orthonormal basis of $\RR^5$ such that $P = \spn_{\CC}\{e_1 + \ii e_2, e_3 + \ii e_4\}$ then $\pi(P) = e_0$.  We thus have an orthogonal direct sum decomposition
$$
\CC^5 = P \oplus \ov{P} \oplus  \spn_{\CC}\,e_0 \,.
$$

To identify the vertical and horizontal spaces of the Riemannian submersion \eqref{twistor-proj-T},
recall first that the $(1,0)$-tangent space of $G_m(\CC^n)$ at $P \in G_m(\CC^n)$ may be identified with the space $\Ll(P, \CC^n/P) \cong \Ll(P,P^{\perp})$ of complex linear maps from $P$ to $\CC^n/P \cong P^{\perp}$, see for example \cite [\S 2A]{EeWo}.   Then it follows from \eqref{quadric-Grass} that, for $P \in \Tt_{m,n}$, we have 
$$
T'_P\Tt_{m,n} = \{\ell \in \Ll(P,P^{\perp}) : \inn{\ell(x), y}^{\cc} + \inn{x,\ell(y)}^{\cc} = 0 \  \text{ for all } x,y \in P \}\,.
$$
It is easy to see that the decomposition \eqref{HV-Si'} into $(1,0)$- vertical and horizontal spaces at $P \in \Tt_{2,5}$ is given by $T'_P\Tt_{2,5} = \Vv'_P \oplus \Hh'_P$
where
\begin{equation} \label{HV-P}
	\left\{\begin{array}{rcl}
	\Vv'_P &= &\{\ell \in \Ll(P,\ov{P}) : \inn{\ell(x), y}^{\cc} + \inn{x,\ell(y)}^{\cc} = 0 \}, \\
	\Hh'_P &= &\Ll(P,  \spn_{\CC}{e_0}).
	\end{array} \right.
\end{equation}

\smallskip

2) \emph{$\CP^3$ as twistor space.} Let $\HH = \{a + b\jj : a,b \in \CC\}$ denote the skew-field of quaternions.  The map $a +b\jj \mapsto (a,b)$ gives a canonical identification of $\HH$ with $\CC^2$ and so of $\HH^2$ with $\CC^4$. Let $\HP^1$ be quaternionic projective space consisting of all one-dim\-en\-sion\-al quaternionic subspaces of $\HH^2$; thus $\HP^1$ is the quotient of $\HH^2 \setminus \{\vec 0\}$ by the (left)-action by $\HH \setminus \{0\}$. We identify $\HP^1$ with $S^4$ by formula \eqref{CP1S2} with the $z_i$ in $\HH$. Then we have the celebrated \emph{Calabi--Penrose twistor map}
\begin{equation} \label{twistor-CP3}
\pi: \CP^3 \to S^4
\end{equation}
given by mapping a complex one-dim\-en\-sion\-al subspace $\spn_{\CC}{v} \in \CP^3$ to the unique quaternionic one-dim\-en\-sion\-al subspace of $\HH^2 = \CC^4$ which contains it; explicitly, $\pi(\spn_{\CC}v) = \spn_{\HH}v = \spn_{\CC}\{v, \jj v\}$.   Give $\CP^3$ its standard Fubini--Study metric and standard complex structure $J$ so that it becomes a K\"ahler manifold and $\pi$ becomes a Riemannian submersion (up to scale).  Then \eqref{twistor-CP3} is another realization of the twistor bundle $\pi:\Si^+(S^4) \to S^4$ with the decompositions \eqref{HV-Si}, \eqref{HV-Si-C} and \eqref{HV-Si'} reading
\begin{equation} \label{HV-CP3}
T\CP^3 = \Hh \oplus \Vv, \quad T^{\cc}\CP^3 = \Hh^{\cc} \oplus \Vv^{\cc}, \quad T'\CP^3 = \Hh' \oplus \Vv',
\end{equation}
where $\Hh$ is the orthogonal complement of $\Vv$ with respect to the Fubini--Study metric on $\CP^3$.

\smallskip

It is easy to verify that both the above models have nowhere integrable horizontal spaces, so satisfy the hypotheses of Proposition \ref{prop:Kahler-twistor-sp}.   They can thus be used as models for the twistor space $\Si^+(S^4)$ of $S^4$.
We shall need the following further properties of that twistor space; to prove these, it is convenient to use its realization as $\CP^3$ discussed above.  
 
Define a tensor field $\tilde{\Aa}^{\Hh} \in \Ga(T^*\CP^3 \otimes \Hh^* \otimes \Vv)$ by $\tilde{\Aa}^{\Hh}_X Y = \Vv(\na^{\CP^3}_X Y)$ \ $(X \in \Ga(T\CP^3), \ Y \in \Ga(\Hh))$.  This extends by complex linearity to a section of $T^{\cc}_*\CP^3 \otimes \Hh^{\cc}_* \otimes \Vv^{\cc}$ which we continue to denote by $\tilde{\Aa}^{\Hh}$.    The tensor $\tilde{\Aa}^{\Hh}$ has the following properties.

\begin{lem} \label{lem:hor-sbb}
{\rm (i)} For any $X \in \Ga(T^{\cc}\CP^3)$, $\tilde{\Aa}_X^{\Hh}$ respects $J$, i.e., $\tilde{\Aa}^{\Hh}_X(JY) = J \tilde{\Aa}^{\Hh}_X Y$ for all\/ $Y \in \Ga(\Hh^{\cc})$.

{\rm (ii)} $\tilde{\Aa}^{\Hh}_X$ is zero for all $X \in \Vv^{\cc}$.

{\rm (iii)} The $(1,0)$-horizontal subbundle $\Hh'$ is a \emph{holomorphic subbundle} of $T'\CP^3$; equivalently $\tilde{\Aa}^{\Hh}_{\ov X}Y = 0$\/ for all $X \in \Ga(T'\CP^3)$, $Y \in \Ga(\Hh')$.
\end{lem}

\begin{proof}
(i) This follows quickly from the parallelity of $J$.

(ii)  Let $X,\, Z \in \Ga(\Vv^{\cc})$ and $Y \in \Ga(\Hh^{\cc})$.  Then we have $\inn{\Aa^{\Hh^{\cc}}_X Y, Z} = \inn{\na^{\CP^3}_X Y, Z} = -\inn{Y, \na^{\CP^3}_X Z}=0$, since the fibres of $\pi$ are totally geodesic.

(iii) Let $X, Y \in \Ga(\Hh^{\cc})$.  Since $\pi$ is a Riemannian submersion, $\tilde{\Aa}^{\Hh}$ is antisymmetric on horizontal vectors \cite{O'N}.  Thus, using part (i),
\begin{equation} \label{J-linear}
\tilde{\Aa}^{\Hh}_{JX} Y = - \tilde{\Aa}^{\Hh}_Y(JX) = -J \tilde{\Aa}^{\Hh}_Y X = J \tilde{\Aa}^{\Hh}_X Y \,;
\end{equation}
together with (ii) this gives (iii).
\end{proof}

Denote the restriction of $\tilde{\Aa}^{\Hh}$ to $\Hh^{\cc}$ by $\Aa^{\Hh} \in \Ga(\Hh^{\cc}_* \otimes \Hh^{\cc}_* \otimes \Vv^{\cc})$.  This is essentially O'Neill's tensor \cite{O'N}; we shall call it the \emph{second fundamental form of\/ $\Hh$ (in\/ $T\CP^3$)}.  This tensor has the following properties.

\begin{lem} \label{lem:sff}
{\rm (i)}  $\Aa^{\Hh}$ is \emph{antisymmetric}, i.e., $\Aa^{\Hh}_X Y = -\Aa^{\Hh}_Y X$ for all $X, Y \in \Ga(\Hh^{\cc})$.
Hence $\Aa^{\Hh}_X Y = \frac{1}{2} I(X,Y)$ \ $(X, Y \in \Ga(\Hh^{\cc})\,)$ where $I$ is the integrability tensor \eqref{int-tensor} of\/ $\Hh$.

{\rm (ii)} $\Aa^{\Hh}$ restricts to a tensor  $\Aa^{\Hh'} \in \Ga\bigl((\Hh')^* \otimes (\Hh')^* \otimes \Vv'\bigr)$ which we shall call the \emph{second fundamental form of $\Hh'$}.

{\rm (iii)} For each $w \in \CP^3$ and non-zero $X \in \Hh'_w$, the linear map $\Aa^{\Hh'}_X:\Hh'_w \to \Vv'_w$ is non-zero.

{\rm (iv)}  $\Aa^{\Hh'}$ is holomorphic in horizontal directions in the sense that
$\na_{\ov{Z}} \Aa^{\Hh'} = 0$ for all $Z \in \Hh'$.  (Here $\na$ denotes the connection on
$\Hh^{\cc}_* \otimes \Hh^{\cc}_* \otimes \Vv^{\cc}$ induced from the Levi-Civita connection $\na^{\CP^3}$ of\/ $\CP^3$.) 
\end{lem}

\begin{proof}
(i)   Antisymmetry follows as for $\tilde{\Aa}^{\Hh}$ above.  Thus
$$
\Aa^{\Hh}_X Y = - \Aa^{\Hh}_Y X = \tfrac{1}{2} \Vv([X,Y]) = \tfrac{1}{2} I(X,Y)\,.
$$

(ii) This quickly follow from \eqref{J-linear}.

(iii) Suppose that there is some $w \in \CP^3$ and non-zero $X \in \Hh'_w$ such that $\Aa^{\Hh'}_X Y = 0$ for all $Y \in \Hh'_w$.  Then $\Aa^{\Hh'}_Y X = -\Aa^{\Hh'}_X Y = 0$.  Choose $Y$ such that $\{X,Y\}$ is a basis for $\Hh'_w$.  Then these equations together with antisymmetry show that $\Aa^{\Hh'}$ is zero.  It follows from this and Lemma \ref{lem:hor-sbb}(iii) that $\Aa^{\Hh}$ is also zero.  However, this is not the case, as $\Hh$ is nowhere integrable by O'Neill's formulae \cite{O'N}.

(iv) Let $X$, $Y$ and $Z$ be holomorphic sections of $\Hh'$.  Then, using the definition of $\Aa^{\Hh'}$, we have
\begin{eqnarray*}
(\na_{\ov{Z}} \Aa^{\Hh'})_X Y &= &\Vv\bigl(\na^{\CP^3}_{\ov{Z}} (\Aa^{\Hh'}_X Y)\bigr) \\
&= &\Vv\bigl(\na^{\CP^3}_{\ov{Z}}(\Vv(\na^{\CP^3} _X Y))\bigr) \\
&= &\Vv\bigl(\na^{\CP^3}_{\ov{Z}}(\na^{\CP^3} _X Y)\bigr)
		- \Vv\bigl(\na^{\CP^3}_{\ov{Z}}(\Hh(\na^{\CP^3} _X Y))\bigr)\,.
\end{eqnarray*}
Now, the last term on the right-hand side is zero since $\Hh(\na^{\CP^3} _X Y))$ lies in $\Hh'$ which, by Lemma \ref{lem:hor-sbb}(iii), is a holomorphic subbundle of $T^{\cc}\CP^3$.  Using the definition of the curvature tensor, the first term on the right-hand side equals
$$		
\Vv\bigl(\na^{\CP^3}_X(\na^{\CP^3} _{\ov{Z}} Y)\bigr) + \Vv(\na^{\CP^3}_{[\ov{Z},X]}Y) + \Vv\bigl(R(X,\ov{Z})Y\bigr) \,.
$$
The first term of this is zero by holomorphicity of $Y$; the second term is zero since $[\ov{Z},X] = \na^{\CP^3}_{\ov{Z}}X - \ov{\na^{\CP^3}_{\ov{X}}Z}$ which vanish by holomorphicity of $X$ and $Z$; and the last term is zero by the standard formula \cite[Chapter 9]{KoNo} for the curvature tensor of $\CP^3$.		
\end{proof}

\subsection{Holomorphic maps into the twistor space} \label{subsec:hol-twistor}

Let $\Hol(S^2,\CP^3)$ denote the set of holomorphic maps from the Riemann sphere $S^2$ to $\CP^3$.  On identifying $S^2$ biholomorphically with the extended complex plane $\{z \in \Ci\}$ as in \eqref{stereo}, we see that any $f \in \Hol(S^2,\CP^3)$ is of the form
\begin{equation} \label{holmap-CP3}
f(z) =  [F(z)] = [F_0(z), F_1(z), F_2(z), F_3(z)] \qquad (z \in \Ci)
\end{equation}
where $F = (F_0, F_1, F_2, F_3)$ is a quadruplet of polynomials which we take to be coprime; we interpret $f(\infty)$ as a limit. (Here, for $W \in \CC^4 \setminus \{\vec{0}\}$,  $[W] = [W_1,W_2,W_3,W_4]$ denotes the point with homogeneous coordinates $(W_1,W_2,W_3,W_4)$ so that $[W] = p(W)$ where $p$ denotes the natural projection from $\CC^4 \setminus \{\vec{0}\}$ to $\CP^3$.)  The \emph{degree} of $f$ is defined to be the degree of the induced mapping on cohomology:   $f^*:\ZZ \cong H^2(\CP^3,\ZZ) \to H^2(S^2,\ZZ) \cong \ZZ$; this is equal to the maximum degree of the polynomials $F_i$\,.

We give $\Hol(S^2,\CP^3)$ the compact-open topology.   Then its connected components are the spaces $\Hol_d(S^2,\CP^3)$ of holomorphic maps of degree $d$.   Let $\CC[z]_d^4$ denote the vector space of all quadruplets of polynomials of degree at most $d$;  taking coefficients defines a canonical isomorphism of  $\CC[z]_d^4$ with $\CC^{4d+4}$.  This factors to a canonical identification of the projectivization $P(\CC[z]_d^4)$ with $\CP^{4d+3}$.

The map $i:\Hol_d(S^2,\CP^3) \to P(\CC[z]_d^4)$ given by $f \mapsto [F_0, F_1, F_2, F_3]$ defines a bijection onto the dense open subset $V_d$ of $P(\CC[z]_d^4) \cong \CP^{4d+3}$ given by those quadruplets of polynomials which are coprime and of maximum degree exactly $d$.  Giving $V_d$ the subspace topology, it is easily seen that this bijection is a homeomorphism which endows $\Hol_d(S^2,\CP^3)$ with the structure of a complex manifold of dimension $4d+3$.

Call a map $f:S^2 \to \CP^3$ \emph{full} if its image does not lie in a proper projective subspace $\CP^2 \subset \CP^3$.  Let $\Hol_d^{\ff}(S^2,\CP^3)$ denote the space of full holomorphic maps of degree $d$.  Now  $i$ maps $\Hol_d^{\ff}(S^2,\CP^3)$ onto the open subset $V_d^{\ff}$ of $V_d$ given by those quadruplets $F \in V_d$ whose components $F_i$ are linearly independent.  Since any four polynomials of degree $\leq 2$ are linearly dependent, $V_d^{\ff}$ is non-empty if and only if $d \geq 3$. Hence, \emph{$\Hol_d^{\ff}(S^2,\CP^3)$ is empty if\/ $d \leq 2$ and is a complex manifold of dimension $4d+3$ if\/ $d \geq 3$}.

A smooth map $f:S^2 \to \CP^3$ is called \emph{horizontal} if the image of its differential is contained in the horizontal subbundle $\Hh$ given by \eqref{HV-CP3}.  Let $w = [W]$ where $W = (W_0,W_1,W_2,W_3)$; then, on identifying $T'_w\CP^3$ with  $\Ll(w,w^{\perp})$, the $(1,0)$-vertical space $\Vv'_w$ at $w$ is the subspace $\Ll \bigl(w,\spn(\jj\, W)\bigr)$ where $\jj\, W = (-\ov{W}_1,\ov{W}_0,-\ov{W}_3,\ov{W}_2)$.  A \emph{holomorphic} map \eqref{holmap-CP3} is horizontal if and only if the image of $\dd f$ is Hermitian orthogonal  to $\Vv'$ in $T'\CP^3$.  To write this nicely, let $J_0$ denote the matrix
$$
J_0 = \left( \begin{array}{cccc} 0& -1 & 0 & 0 \\ 1 & 0 & 0 & 0 \\ 0 & 0 & 0 & -1 \\ 0 & 0 & 1 & 0 \end{array} \right);
$$
then \emph{a vector $v = \dd p_W(V)$ at $w$ is horizontal if $\inn{V, J_0W}^{\cc} = 0$}.

Note that the group of holomorphic isometries of $\CP^3$ is the projective group $\PU(4)$ corresponding to $\U(4)$.  The subgroup of $\U(4)$ which preserves $J_0$ and hence the horizontality condition, is the symplectic group $\Sp(2)$, cf.\ \cite{BoWo-higher}. 

For any $d \in \{1,2, \ldots \}$, define a map $Q =  Q_d:\CC[z]_d^4 \to \CC[z]_{2d-2}$ by
\begin{eqnarray}
Q(F) = Q_d(F) &=  &\inn{F', J_0 F} \ = \ \{F_0,F_1\}+\{F_2,F_3\} \notag \\
		&= &F_1 F'_0 - F_0 F'_1 + F_3 F'_2 - F_2 F'_3\,. \label{Q}
\end{eqnarray}

Here, for any polynomials $G_1$, $G_2$, we write $\{G_1,G_2\} = -G_1 G_2' + G_1' G_2$.  Note that
$\deg\{G_1,G_2\} = \deg G_1 + \deg G_2 - 1$ unless $\deg G_1 = \deg G_2$, in which case $\deg\{G_1,G_2\} \leq \deg G_1 + \deg G_2 - 2$.    Then \emph{$f = [F] \in \Hol_d(S^2,\CP^3)$ is horizontal if and only if}
\begin{equation} \label{horizontal}
Q(F) = 0 \,.
\end{equation}
Note that this condition makes sense for any choice of holomorphic lift $F$ of $f$ defined on an open subset of the domain and not necessarily polynomial, and is independent of that choice.

Thus, under the inclusion mapping $i:\Hol_d(S^2,\CP^3) \to \CP^{4d+3}$, the space of horizontal holomorphic maps $\HHol_d(S^2,\CP^3)$ of degree $d$ corresponds to the intersection of the algebraic variety $Z(Q_d) = \{[F] \in P(\CC[z]_d^4) : Q_d(F) = 0\}$ with the dense open subset $V_d$ of $P(\CC[z]_d^4)$.  The subspace $\HHol_d^{\ff}(S^2,\CP^3)$ of \emph{full} horizontal holomorphic maps corresponds to the intersection of $Z(Q_d)$ with the smaller dense open subset $V_d^{\ff}$ of $P(\CC[z]_d^4)$.

Given a holomorphic map $f:M^2 \to \CP^3$, we form the pull-back bundles $V' = f^{-1}(\Vv')$ and $H' = f^{-1}(\Hh')$ over $M^2$.  They are complex subbundles of the holomorphic bundle $f^{-1}T'\CP^3$.   For any complex coordinate $z$ for $M^2$, define linear bundle maps 
$A'_{H'}\,, A''_{H'}:H' \to V'$ locally by
$$
A'_{H'}(Y) = V(\na^f_{\pa/\pa z}Y) \,, \ 
	A''_{H'}(Y) = V(\na^f_{\pa/\pa \bar{z}}Y) \qquad \bigl(Y \in \Ga(H')\bigr).	
$$
Here $V$ denotes projection associated to the pull-back of the decomposition of complexified bundles \eqref{HV-Si-C}; parallelity of $J$ as in Lemma \ref{lem:hor-sbb}(i) ensures that the maps $A'_{H'}$ and $A''_{H'}$ have image in $V'$.    Similar definitions can be given for $A'_{V'}\,, A''_{V'}:V' \to H'$ by reversing the roles of $H$ and $V$.   It is easy to see that $A'_{V'}$ and $A''_{V'}$ are minus the adjoints of $A''_{H'}$ and  $A'_{H'}$, respectively.
Note that we have used a complex coordinate for convenience; our results will not depend on that choice.

Recall that, if $M^2$ is a Riemann surface, any complex bundle $E \to M^2$ equipped with a connection $\na^E$ has a unique holomorphic structure such that a section $s$ is holomorphic if and only if $\na^E_{\ov{Z}}s = 0$ \ $(Z \in T'M^2)$ \cite{KoMa}; this is the \emph{Koszul--Malgrange holomorphic structure}.
When $E = T'M^2$, this coincides with the natural holomorphic structure on $T'M^2$.

Say that a smooth map $f:M^2 \to \CP^3$ is \emph{vertical} if its image lies in a fibre of the twistor projection \eqref{twistor-CP3}, equivalently,   $\pi \circ f$ is a constant map.  
We have a version of Proposition \ref{lem:sff} for the pull-back bundles $H'$ and $V'$ as follows.

\begin{prop} \label{prop:hor-sbb2}

Let $f:M^2 \to \CP^3$ be a holomorphic map.  

{\rm (i)} $H'$ is a holomorphic subbundle of $f^{-1}T'\CP^3$; equivalently, $A''_{H'} = 0$.

{\rm (ii)} Suppose that $f$ is horizontal.  Then $A'_{H'}$ is holomorphic with respect to the Koszul--Malgrange holomorphic structure on $H'^* \otimes V'$, i.e., $\na^{H'^* \otimes V'}_{\pa/\pa \bar{z}}(A'_{H'}) = 0$;
equivalently,  $A''_{V'}$ is antiholomorphic in the sense that $\na^{V'^* \otimes H'}_{\pa/\pa z}(A''_{V'}) = 0$.

{\rm (iii)} Suppose that $f$ is not vertical. Then $A'_{H'}$ is not identically zero; equivalently $A''_{V'}$ is not identically zero.
\end{prop}

\begin{proof}

(i) It is the pull-back of a holomorphic subbundle by a holomorphic map.

(ii) {}From its definition, we see that that $A'_{H'}$ is the pull back of $\Aa^{\Hh'}_Z$ where $Z = \dd f(\pa/\pa z)$, so the result follows from Lemma \ref{lem:sff}(iv).   Taking the adjoint yields $\na^{V'^* \otimes H'}_{\pa/\pa z}(A''_{V'}) = 0$.

(iii) For any $x \in M^2$, on identifying $H'_x$ with $\Hh'_{f(x)}$, we have $A'_{H'} = \tilde{\Aa}^{\Hh'}_Z$, where $Z = \dd f_x(\pa/\pa z) \in T_{f(x)}'\CP^3$.   By Lemma \ref{lem:hor-sbb}(ii), this equals $\Aa^{\Hh'}_{\Hh(Z)}$.  Choose $x$ to be a point where $f$ is not vertical so that $\Hh(Z)$ is non-zero.  Then $\Aa^{\Hh'}_{\Hh(Z)}$, and so $A'_{H'}$, is non-zero by Lemma \ref{lem:sff}(iii).
\end{proof}

\subsection{Infinitesimal deformations of holomorphic maps} \label{subsec:inf-hol}
Let $f:M \to N$ be a holomorphic map between complex manifolds.  For any $u \in \Ga(f^{-1}TN)$, let $u'$
denote its \emph{$(1,0)$-component} under the decomposition $f^{-1}T^{\cc}N = f^{-1}T'N \oplus f^{-1}T''N$; thus $u' = \tfrac{1}{2}(u - \ii J^N u)$.  Say that $u$ (or $u'$) is \emph{holomorphic} if $\na^f_{\ov{Z}} u' = 0$ for all
$Z \in T'M$.  We have the following analogue of Proposition \ref{prop:first-order} for holomorphic maps (see the footnote to that proposition for the meaning of $\pa/\pa t$).

\begin{defn} Let $\{f_t\}$ be a smooth $1$-parameter family of maps between complex manifolds $M$ and $N$.
Say that $\{f_t\}$ is \emph{holomorphic to first order} if
\begin{equation} \label{first-order-hol}
\left(\frac{\pa f}{\pa\ov{z}}\right)' = 0 \quad \text{and} \quad
	\left. \frac{\pa}{\pa t}\left(\frac{\pa f_t}{\pa\ov{z}}\right)' \right\vert_{t=0} = 0 \,.
\end{equation}
\end{defn}

\begin{prop} \label{prop:inf-hol}
 Let\/ $f:M \to N$ be a holomorphic map between K\"ahler manifolds, and let\/ $u \in \Ga(f^{-1}TN)$. 
Let $\{f_t\}$ be a smooth variation of\/ $f$ tangent to $u$.  Then
$$
\na^f_{\ov{Z}}u' =  \left. \frac{\pa}{\pa t}\left(\frac{\pa f_t}{\pa\ov{z}}\right)' \right\vert_{t=0} \,.
$$

In particular, $u$ is a holomorphic vector field along $f$ if and only if $\{f_t\}$ is holomorphic to first order.
\end{prop}

\begin{proof}
Working on $M \times (-\ep,\ep)$, we swap the order of the derivatives and use the parallelity of $J$, see \cite[Proposition 4.1]{Wo-Tokyo}.
\end{proof}

For the following see \cite[Chapter 5, Theorem (3.2)]{Ur}, \cite[Proposition 4.1]{LeWo2} or \cite[Proposition 4.2]{Wo-Tokyo}.

\begin{prop} \label{prop:hol}
Let\/ $f:M \to N$ be a holomorphic map between compact K\"ahler
manifolds.
Then any Jacobi field along\/ $f$ is holomorphic.
\qed \end{prop}

Note that, when $M^2$ is a Riemann surface, a vector field $u$ along $f$ is holomorphic if, in any local coordinate $z$ on $M^2$,
\begin{equation}
\na^f_{\pa/\pa\bar{z}} u' = 0 \,.
\end{equation}

Let $n \in \{1,2,\ldots\}$.  For a holomorphic map $f:M^2 \to \CP^n$ we have the following description. Let $F:U \to \CC^{n+1} \setminus \{\vec{0}\}$ be a holomorphic lift of $f$ defined on some open subset of $M^2$, so that $f = [F] = p \circ F$ where $p:\CC^{n+1} \setminus \{\vec{0}\} \to \CP^n$ is the natural projection.  Then a vector field $u$ along $f$ is given locally by $\dd p_F(\U)$ for some map $\U:U \to \CC^{n+1}$;  more invariantly, its $(1,0)$-part $u'$ is given by the section of $\Ll(F, \CC^{n+1}/F)$ defined by $F \mapsto \U \mod F$. The vector field $u$ is holomorphic if and only if we can choose $\U$ to be holomorphic.    If $f \in \Hol_d(S^2,\CP^3)$ we may take $F$ and $\U$ to have components polynomial of degree $\leq d$. 

\begin{prop}  \cite[Proposition 4.2]{LeWo2} \label{prop:int-hol}
Any holomorphic vector field $u$ along a holomorphic map $f:S^2 \to \CP^n$ is \emph{integrable by holomorphic maps}, i.e., there is a smooth one-parameter family of holomorphic maps $f_t:S^2 \to \CP^n$ with  $f_0 = f$ and $\pa f_t/\pa t|_{t=0} = u$.
\qed \end{prop}

Indeed, we can take $u$ to be given by rational functions; it can then be integrated explicitly.

\begin{defn} \label{def:inf-HH}
Let $f:M^2 \to \CP^3$ be a smooth map which is horizontal. Say that a vector field $u$ along $f$ is an \emph{infinitesimal horizontal deformation} if it preserves horizontality to first order in the sense that, for any smooth one-parameter variation $\{f_t\}$ of $f$ tangent to $u$, the norm of the vertical component of $\dd\! f_t$ is $o(t)$.  If, additionally, $f$ and $u$ are holomorphic, we call $u$ an \emph{infinitesimal horizontal holomorphic deformation (IHHD) of\/ $f$}.
\end{defn}

N.\ Ejiri \cite[Sec.\ 2]{Ej-min} gives an equation for IHHDs which shows that this notion is independent of the choice of $\{f_t \}$.

\smallskip

By differentiating the formula \eqref{Q}, we see that a holomorphic vector field $u = \dd p_F(\U)$ along a holomorphic map $f = p \circ F$ is an IHHD if
 \begin{equation} \label{inf-hor}
\dd Q_F(\U) \equiv \{\U_0,F_1\} + \{F_0,\U_1\} + \{\U_2, F_3\} + \{F_2, \U_3\} = 0 \,.
\end{equation}

Now the differential $\dd Q_F$ is a linear map from $\CC[z]_d^4$ to $\CC[z]_{2d-2}$, and the space of solutions to \eqref{inf-hor} has dimension $\dim\ker\dd Q_F = 4d+4 - \dim\image\;\dd Q_F$.  Since $\dim\image\;\dd Q_F \leq 2d-1$ this is at least $2d+5$.   One solution is $\U = F$ which projects to $u=0$, so the dimension of the space of IHHDs is at least $2d+4$.

We now consider what we can say about $\HHol_d(S^2,\CP^3)$.  
J.-L.\ Verdier \cite{Ve2} showed that (i) $\HHol_d(S^2,\CP^3)$ is a connected algebraic variety of pure dimension $2d+4$;
(ii) when $d=1,2$,  $\HHol_d(S^2,\CP^3) = \HHol_d^{\nf}(S^2,\CP^3)$, the subspace of non-full maps, and this is irreducible; (iii) when $d \geq 3$, $\HHol_d(S^2,\CP^3)$ has the two irreducible components $\ov{\HHol_d^{\ff}(S^2,\CP^3)}$ and $\HHol_d^{\nf}(S^2,\CP^3)$.  

\begin{defn} \label{def:regular-ho}  
A point $f = [F] \in \HHol_d(S^2,\CP^3)$ is called \emph{(a) regular (point of\/ $Q$)} if $Q$ is submersive at $F$; this condition is clearly independent of the representative $F$ chosen.  
\end{defn}

By the Inverse Function Theorem, if $f$ is regular, then it is a \emph{smooth point} of $\HHol_d(S^2,\CP^3)$, i.e., a point where that algebraic variety is a smooth manifold.  The converse holds for non-full maps, see Proposition \ref{prop:jn-pt}.   The following shows that there are regular points in each irreducible component, thus confirming that $\HHol_d(S^2,\CP^3)$ has pure dimension $2d+4$.

\begin{lem} \label{lem:Q-subm}
The following maps are regular:

{\rm (i)} $f = [F] \in \HHol_d^{\ff}(S^2,\CP^3)$, where\\[-3ex]
	$$F(z) = [z^{d-1}-1\,,\, -d\,z/(d-2)\,,\, z^d - d\,z/(d-2)\,,\, 1] \quad (d \geq 3) \,;$$

{\rm (ii)} $f = [F] \in \HHol_d^{\nf}(S^2,\CP^3)$, where $F(z) = [z^d,0,1,0]$ \ $(d \geq 1)$.
\end{lem}

\begin{proof}
(i) We calculate $\dd Q_F(0,0,\U_2,0) = \{\U_2,1\} = \U_2'$; on putting $\U = z,\ldots, z^d$, this gives multiples of $1,\ldots, z^{d-1}$, respectively.
On the other hand, $\dd Q_F(\U_0,0,0,0) = \{z^{d-1}-1, \U_1 \}$ on putting $\U_1 = 1,\ldots, z^{d-2}$, this gives polynomials of order $d-2$, \ldots, $2d-1$, respectively.  
It is clear from this that $\dd Q_F$ is surjective.   (ii) This is similar.
\end{proof}

Let $f:M^2 \to \CP^3$ be horizontal and holomorphic.  Say that an IHHD $u$ along $f$ is \emph{integrable by horizontal holomorphic maps} if it is tangent to a deformation of $f$ by horizontal holomorphic maps,  i.e., $u = \pa f_t/\pa t|_{t=0}$ for some one-parameter family of horizontal holomorphic maps $f_t:S^2 \to \CP^3$ with $f_0 = f$.   

\begin{prop} \label{prop:smooth}
Let $f = [F] \in \HHol_d(S^2,\CP^3)$. Then the following are equivalent:

{\rm (i)} $f$ is a regular point;

{\rm (ii)} the dimension of the space of IHHDs at $f$ is $2d+4$;

{\rm (iii)} all IHHDs of $f$ are integrable by horizontal holomorphic maps.
\end{prop}

\begin{proof}
If  $f$ is a regular point then, as above, the space of IHHDs has dimension $2d+4$, and so coincides with the tangent space to  $\HHol_d(S^2,\CP^3)$ at $f$.  Hence all IHHDs are integrable.

Conversely, if $f$ is not regular, then the space of IHHDs at $f$ has dimension more than $2d+4$.  Since  $\HHol_d(S^2,\CP^3)$ is an algebraic variety of dimension $2d+4$, not all the IHHDs in that space can be integrable.
\end{proof}

\begin{cor} \label{cor:smooth}
If\/ $f$ is a non-smooth point of\/ $\HHol_d(S^2,\CP^3)$, then there are IHHDs of $f$ which are not integrable by horizontal holomorphic maps. \qed
\end{cor}

\subsection{The twistor correspondence}

Let $\Harm(S^2,S^4)$ denote the set of harmonic maps from the $2$-sphere $S^2$ to $S^4$, equipped with the compact-open topology. A map to $S^4$ is called \emph{full} if it does not have image in a totally geodesic (i.e., equatorial) $S^3 \subset S^4$.

The twistor projection \eqref{twistor-S4} gives the following well-known correspondence beween horizontal holomorphic maps into the twistor space and harmonic maps into $S^4$, see, for example \cite{EeSal,Sal}; we shall generally use the $\CP^3$-model \eqref{twistor-CP3} for the twistor bundle.

Given a smooth map $\phi:M^2 \to S^4$, as before, write $\varPhi = i \circ \phi$ where $i:S^4 \to \RR^5$ is the canonical inclusion.  For any local complex coordinate $z$ on $U \subset M^2$, define a function $W(\phi):U \to \CC$ by
$$
\varPhi \wedge \frac{\pa\varPhi}{\pa z} \wedge \frac{\pa^2\varPhi}{\pa z^2}
	\wedge \frac{\pa\varPhi}{\pa\bar{z}} \wedge \frac{\pa^2\varPhi}{\pa\bar{z}^2} = W(\phi) \,\omega
	\,.
$$
Here $\omega \in \wedge^5 \CC^5$ is the volume form of $\RR^5$ with its standard orientation.  
Then, for any $x \in U$, the number $W(\phi)(x)$ is real; its sign is clearly independent of the choice of complex coordinate and is called the \emph{spin of\/ $\phi$ at $x$}.   If $M^2 = S^2$ and $\phi$ is harmonic, $W(\phi)$ has fixed sign on $M^2$ and we have three cases  \cite{Ve2}:

(i) $\phi$ is not full and its spin is identically zero, then $\phi$ has image in a totally geodesic (equatorial) $S^2\subset S^4$: we say that $\phi$ has \emph{zero spin};

(ii) $\phi$ is full and its spin is strictly positive, except at isolated points where it is zero: we say that $\phi$ has \emph{positive spin};

(iii) $\phi$ is full and its spin is strictly negative, except at isolated points where it is zero: we say that $\phi$ has \emph{negative spin}.

Thus the space $\Harm(S^2,S^4)$ is the union of the three disjoint spaces $\Harm^{\nf}(S^2,S^4)$, $\Harm^+(S^2,S^4)$ and $\Harm^-(S^2,S^4)$ of harmonic maps of zero, positive and negative spin, respectively.    Write $\Harm^{\geq 0}(S^2,S^4) = \Harm^{\nf}(S^2,S^4) \cup \Harm^+(S^2,S^4)$.  Note that composing $\phi$ with the antipodal map $S^4 \to S^4$, \ $x \mapsto -x$, changes the sign of its spin and defines a homeomorphism between $\Harm^+(S^2,S^4)$ and $\Harm^-(S^2,S^4)$.  In particular, for any harmonic map $\phi:S^2 \to S^4$, one of $\pm\phi$ lies in $\Harm^{\geq 0}(S^2,S^4)$.

The real isotropy \eqref{isotropy} of a harmonic map from $S^2$ to $S^4$ tells us that, at each point $x \in S^2$ where it is non-zero, $\ds\frac{\pa\varPhi}{\pa\bar{z}} \wedge \frac{\pa^2\varPhi}{\pa\bar{z}^2}$ defines an isotropic subspace $P$ of $T^{\cc}_x S^4$.   Note that \emph{$\phi$ is of positive (resp.\ negative) spin if and only if\/ $P$ is a positive (resp.\ negative) isotropic subspace} (see \S \ref{subsec:twistor-S4} above for the definition of positivity).

\begin{prop} \label{prop:twistor-bij}
The map $f \mapsto \phi = \pi \circ f$ defines homeomorphisms:
\begin{eqnarray}
\Pi^{\hor}:\HHol(S^2,\CP^3) \to \Harm^{\geq 0}(S^2,S^4) \label{Pi-hor} \,, \\
\Pi^{\hor, \ff}:\HHol^{\ff}(S^2,\CP^3) \to \Harm^+(S^2,S^4) \label{Pi-hor-ff} \,, \\
\Pi^{\hor, \nf}:\HHol^{\nf}(S^2,\CP^3) \to \Harm^{\nf}(S^2,S^4) \label{Pi-hor-nf} \,. 
\end{eqnarray}

Further, the energy of\/ $\phi$ is equal to $d$ times the area of\/ $S^2$ $= 4\pi d$ where $d \in \{0,1,2,\ldots\}$ is the degree of $f$.  
\end{prop}

Since they are induced by the twistor projection \eqref{twistor-CP3}, we shall also call the above maps or their restrictions \emph{twistor projections}.  The integer $d$ is called \cite{EeSal} the \emph{twistor degree} of $\phi$.  Since it is weakly conformal, \emph{a harmonic map $S^2 \to S^4$ with twistor degree $d$ has both energy and area equal to $4\pi d$.}
Now, the components of the space $\HHol(S^2,\CP^3)$ are the spaces $\HHol_d(S^2,\CP^3)$ of horizontal holomorphic maps of degree $d$ \ $(d = 0,1,2,\ldots)$.  It follows that the spaces $\Harm_d(S^2,S^4)$ of (full and non-full) harmonic maps from $S^2$ to $S^4$ of twistor degree $d$ \ $(d = 0,1,2,\ldots)$ form the connected components of $\Harm(S^2,S^4)$  \cite{Lo-S4,Ve3}. 

For any $d$, the space \,$\Harm_d(S^2,S^4)$\, is the union of the three disjoint subspaces \,$\Harm^{\nf}_d(S^2,S^4)$, \,$\Harm^+_d(S^2,S^4)$\, and \,$\Harm^-_d(S^2,S^4)$\, of harmonic maps of zero, positive and negative spin, respectively.   We write \,$\Harm_d^{\geq 0}(S^2,S^4) = \Harm^{\nf}_d(S^2,S^4) \cup \Harm^+_d(S^2,S^4)$\,, then we have

\begin{cor} \label{cor:twistor-bij}
For each\/ $d \in \{1,2,\ldots\}$, the maps \eqref{Pi-hor}---\eqref{Pi-hor-nf} restrict to homeomorphisms
\begin{eqnarray}
\Pi_d^{\hor}:\HHol_d(S^2,\CP^3) \to \Harm_d^{\geq 0}(S^2,S^4) \label{Pi-hor-d} \,, \\
\Pi_d^{\hor, \ff}:\HHol_d^{\ff}(S^2,\CP^3) \to \Harm_d^+(S^2,S^4) \label{Pi-hor-ff-d}\,, \\
\Pi_d^{\hor, \nf}:\HHol_d^{\nf}(S^2,\CP^3) \to \Harm_d^{\nf}(S^2,S^4) \label{Pi-hor-nf-d}\,.
\end{eqnarray}
In particular, $\Harm_d^{\ff}(S^2,S^4)$ is non-empty if and only if\/ $d \geq 3$.
\end{cor}

It follows that, for any $d \geq 3$, the spaces $\Harm^+_d(S^2,S^4)$\, and \,$\Harm^-_d(S^2,S^4)$ are connected, and are the components of \,$\Harm^{\ff}_d(S^2,S^4)$.

For later use, we sketch the proof of the proposition and corollary.   Given $\phi \in \Harm_d^{\geq 0}(S^2,S^4)$, we define its twistor lift $f:S^2 \to \Si^+ = \Si^+(S^4)$ as follows.  At a point $x \in S^2$ where $\phi$ is immersive, set 
$\tau_x M = \dd\phi_x(T_xS^2)$ with orientation chosen such that $\dd\phi_x$ is orientation preserving, and set
\begin{equation}  \label{twistor-lift}
f(x) = \left\{\begin{array}{l}
\!\!\text{\rm the point of $\Si^+_{\phi(x)}$ representing the unique} \\
\!\!\text{\rm positive almost Hermitian structure $J^f_x$ on $T_{\phi(x)}S^4$} \\
\!\!\text{\rm which is rotation through $+\pi/2$ on $\tau_x M$.}
\end{array} \right.
\end{equation}
More explicitly, we have an orthogonal decomposition $T_{\phi(x)}S^4 = \tau_x M \oplus \nu_x M$ into oriented subspaces, and $J^f_x$ is rotation through $+\pi/2$ on both $\tau_x M$ and $\nu_xM$.

Choose any complex coordinate $z$.  Then $\tau_xM$ is determined by the complex vector $\pa\phi/\pa z$; indeed, it is spanned by its real and imaginary parts.  Now, harmonicity \eqref{harmonicity} of $\phi$ tells us that this is a holomorphic section of $\phi^{-1}T^{\cc}S^4$; so an easy argument shows that we can extend $\tau_xM$ smoothly over the critical points of $\phi$, giving a decomposition of smooth bundles:
$$
\phi^{-1}TS^4 = \tau M \oplus \nu M.
$$
Hence $f$ extends smoothly to a map $f:S^2 \to  \Si^+(S^4)$ called the \emph{(positive) twistor lift of\/ $\phi$}.
This map defines a decomposition
\begin{equation} \label{decomp-t}
\phi^{-1}T^{\cc}S^4 = \ov{P} \oplus P
\end{equation}
which, at each point $x \in S^2$, gives the $(1,0)$- and $(0,1)$- tangent spaces of $J^f_x$.  Thus, \emph{in the quadric Grassmannian model $\Tt_{2,5}$ of the twistor space, $f$ is represented by the map $P:S^2 \to \Tt_{2,5}$}\,. 

That $\phi$ has non-negative spin is equivalent to saying that $\pa^2\phi/\pa z^2$ is a section of $P$.
Then it is easy to see that $f$ (equivalently, $P$) is horizontal and holomorphic, cf.\ \S \ref{sub:Jacobi2}.
Further, the map $\phi \mapsto f$ is the inverse of $\Pi_d^{\hor}$.  Since $f$ is horizontal and $\pi$ is a Riemannian submersion, the energy of $\phi$ is equal to the energy of $f$; since $f$ is holomorphic, this equals the area of $S^2$ times the degree of $f$; thus $\Pi_d^{\hor}$ is a bijection from $\HHol_d(S^2,\CP^3)$ to $\Harm_d^{\geq 0}(S^2,S^4)$.   It can be checked that $f$ is full if and only if $\phi$ is full so that $\Pi_d^{\hor}$ restricts to bijections $\Pi_d^{\hor,\ff}$ and $\Pi_d^{\hor,\nf}$.  That all maps are continuous with respect to the compact-open topology is clear; that they are proper follows as in \cite{BoWo-space} or from Theorem \ref{th:smooth} below; thus they are homeomorphisms, completing the proof of the proposition and corollary.

\smallskip

Note that, since $\phi$ is weakly conformal, for any $x \in S^2$, the differential $\dd\phi_x:T_xS^2 \to T_{\phi(x)}S^4$ intertwines the standard almost complex structure on $S^2$ at $x$  and $J^f_x$.

\section{Main results} \label{sec:main}

\subsection{The smoothness of the twistor map}
We now consider our spaces of holomorphic and harmonic maps as subspaces of suitable manifolds of maps.  For simplicity, as in \cite{LeWo1}, we consider spaces of $C^k$ maps for some fixed $k \geq 2$, though many other choices are possible.  The space $C^k(S^2,\CP^3)$ of all $C^k$ mappings from $S^2$ to $\CP^3$ with the compact-open topology forms an infinite-dimensional complex-analytic Banach manifold.  The subset $\Hol(S^2,\CP^3)$ is a closed subspace of that manifold; in fact each component, $\Hol_d(S^2, \CP^3)$, is a closed complex submanifold of dimension $4d+3$.

Similarly, the space $C^k(S^2,S^4)$ of all $C^k$ maps from $S^2$ to $S^4$ forms an infinite-dimensional real-analytic Banach manifold.  The space of harmonic maps $\Harm(S^2,S^4)$ is a closed subspace of that manifold. 

Then the maps \eqref{Pi-hor}--\eqref{Pi-hor-ff-d} are restrictions of the real-analytic map
\begin{equation} \label{Pi}
\Pi:C^k(S^2,\CP^3) \to C^k(S^2,S^4)\,, \quad f \mapsto \phi = \pi \circ f \,.
\end{equation}

The following lemma is proved in \cite[Proposition 2.1 and subsequent remarks]{BoWo-higher}.
\begin{lem} \label{lem:unique} 
Let $\phi:M^2 \to S^4$ be a non-constant smooth map from a Riemann surface.  If there is a  holomorphic map $f:M^2 \to \CP^3$ such that $\pi \circ f = \phi$, then it is unique.
\end{lem}

Say that a vector field $u$ along a smooth map $f:M^2 \to \CP^3$ is \emph{vertical} if $\dd\Pi_f(u) = 0$, i.e, $\dd\pi \circ u$ is identically zero.

\begin{lem} \label{lem:hol-vf}
Let $f:M^2 \to \CP^3$ be a non-vertical holomorphic map.  Then any vertical holomorphic vector field along $f$ is identically zero. 
\end{lem}

\begin{proof}
The $(1,0)$-part $u'$ of such a vector field is a section of $V' = f^{-1}\Vv'$ which satisfies $A''_{V'}u' = 0$.   But by Proposition \ref{prop:hor-sbb2}(iii), unless $f$ is vertical, $A''_{V'}$ is not identically zero, and so is non-zero on a dense open set; it follows that $u \equiv 0$.
\end{proof}

We now consider the image of each component $\Hol_d(S^2,\CP^3)$ under $\Pi$.  To obtain an immersion we shall exclude the closed subspace $\Ver_d \subset \Hol_d(S^2,\CP^3)$ of vertical maps. 

\begin{thm} \label{th:smooth}
For any $d \geq 1$,  the map
\begin{equation} \label{Pi-d-imm}
\Pi_d:\Hol_d(S^2,\CP^3) \setminus \Ver_d \to C^k(S^2,S^4) \setminus \{\text{\rm constant maps}\}
\end{equation}
defined by $f \mapsto \pi \circ f$ is a real-analytic proper injective immersion, so gives a real-analytic diffeomorphism onto a $(4d+3)$-dimensional real-analytic submanifold $R_d$.   This restricts to real-analytic diffeomorphisms:
\begin{eqnarray}
\Pi_d^{\hor}: \HHol_d(S^2,\CP^3) &\to &\Harm^{\geq 0}_d(S^2,S^4) \,, \label{Pi-d-hor} \\
\Pi_d^{\hor,\ff}: \HHol^{\ff}_d(S^2,\CP^3) &\to &\Harm^+_d(S^2,S^4) \,, \label{Pi-d-full-hor} \\
\Pi_d^{\hor,\nf}: \HHol^{\nf}_d(S^2,\CP^3) &\to &\Harm^{\nf}_d(S^2,S^4) \,. \label{Pi-d-non-hor}
\end{eqnarray}
The inverses of these maps are given by taking the twistor lift.

In particular, $\Harm^{\geq 0}_d(S^2,S^4)$ and $\Harm^{\nf}_d(S^2,S^4)$ are real-analytic subvarieties of
the finite-dim\-en\-sion\-al real-analytic manifold $R_d$.
\end{thm}

\begin{proof}
That $\Pi_d$ is \emph{real analytic} is clear; that it is \emph{injective} follows from Lemma \ref{lem:unique}.  That it is \emph{immersive} follows from Lemma \ref{lem:hol-vf}.  

To show that $\Pi_d$ is \emph{proper}, let $\{\phi_n\}$ be a sequence in $R_d$ which converges to $\phi \in R_d$.  Let $f_n = \Pi_d^{-1}(\phi_n) \in \Hol_d(S^2,\CP^3)$.  Then, since $\Hol_d(S^2,\CP^3)$ lies in the compact manifold $\CP^{4d+3}$, a subsequence $\{f_{n_k}\}$ of $\{f_n\}$ must converge in that manifold, let its limit be $f = [F_0,F_1,F_2,F_3]$. The polynomials $F_i$ might not be coprime, i.e., we might have $(F_0,F_1,F_2,F_3) = (q G_0, q G_1, q G_2, q G_3)$ for some polynomials $G_i$ and $q$, with $q$ of degree $\geq 1$ so that $f = [G_0, G_1, G_2, G_3]$.  Now, away from the zeros of $q$, $\{f_{n_k}\}$ tends pointwise to $f$ so that the corresponding subsequence $\{\phi_{n_k}\}$ tends pointwise to $\pi \circ f$ on a dense subset of $S^2$.  However, the original sequence $\{\phi_n\}$ tends to $\phi$ pointwise, so, by the continuity of both maps, $\phi = \pi \circ f$.  Then, by uniqueness (Lemma \ref{lem:unique}), $f$ must be $\Pi_d^{-1}(\phi)$.  Hence $\Pi_d$ is proper.

Noting that a map in $C^k(S^2,S^4)$ has zero energy if and only if it is constant, Proposition \ref{prop:twistor-bij} shows that $\Pi_d$ restricts to the mappings \eqref{Pi-d-hor} and \eqref{Pi-d-non-hor}.
\end{proof}

\begin{rem} \label{rem:amalg}
(i) Of course, $\Hol_d(S^2,\CP^3)$ is a \emph{complex} manifold, and $\HHol_d(S^2,\CP^3)$ and $\HHol^{\nf}_d(S^2,\CP^3)$  are \emph{complex} algebraic subvarieties of it, so that we can transfer those structures to $R_d$, $\Harm^{\geq 0}_d(S^2,S^4)$ and  $\Harm^{\nf}_d(S^2,S^4)$, respectively; we can then give $\Harm_d(S^2,S^4)$ the structure of a complex algebraic variety of pure dimension $2d+4$ by amalgamating the complex algebraic varieties $\Harm^{\geq 0}_d(S^2,S^4)$ and $\Harm^{\leq 0}_d(S^2,S^4)$ along $\Harm^{\nf}_d(S^2,S^4)$, see \cite{Ve2}.   Then, when $d=1$ or $2$,
$\Harm_d(S^2,S^4) = \Harm_d^{\nf}(S^2,S^4)$ is irreducible; when $d \geq 3$, $\Harm_d(S^2,S^4)$ has three irreducible components, namely, $\ov{\Harm_d^+(S^2,S^4)}$, $\ov{\Harm_d^-(S^2,S^4)}$ and $\Harm_d^{\nf}(S^2,S^4)$ \cite{Ve2}. 

However, our result considers these spaces as subspaces of $R_d$ (and so of $C^k(S^2,S^4)$) so that \emph{real}-analyticity is the appropriate notion. 

(ii) All the maps in $R_d$ are real isotropic of positive spin, see \cite{EeSal,Sal}.

(iii) There are families of holomophic maps with image in any given fibre of $\pi$, so that $\Pi_d$ is neither injective nor  immersive when considered as a map from the whole of $\Hol_d(S^2,\CP^3)$.
\end{rem}

\subsection{Correspondence between infinitesimal deformations} \label{sub:Jacobi2}

The following result is non-trivial in the presence of branch points.

\begin{thm} \label{th:Jacobi-lift}
Let $\phi:S^2 \to S^4$ be a full harmonic map of positive spin and let $f:S^2 \to \Si^+ = \CP^3$ be its twistor lift.  The mapping $\dd\Pi_f: u \mapsto v = \dd\pi \circ u$ defines a one-to-one correspondence between the space of infinitesimal horizontal holomorphic deformations $u$ along $f$ and the space of Jacobi fields $v$ along $\phi$.
\end{thm}

\begin{proof}
(a) \emph{Let $u$ be an IHHD along $f$.  We show that $v = \dd\pi \circ u$ is a Jacobi field along $\phi$}.

To do this, let $f_t$ be a one-parameter variation of $f$ tangent to $u$, and set $\phi_t = \pi \circ f_t$. Then $v = \dd\pi \circ u = \pa\phi_t/\pa t|_{t=0}$.  By Proposition \ref{prop:inf-hol},  $f_t$ is holomorphic to first order, and by Definition \ref{def:inf-HH}, it is horizontal to first order, hence
\begin{equation} \label{ft-1st-order}
{\rm (i)} \ \frac{\na}{\pa z}\frac{\pa f_t}{\pa\ov{z}} = o(t)\: \text{ so that } \:\tau(f_t) = o(t)\,; \quad
{\rm (ii)} \ (\dd f_t)^{\Vv} = o(t).
\end{equation}
By the composition law for the tension field \cite{EeSam} we obtain
$$
\tau(\phi_t) = \dd\pi \circ \tau(f_t) + \Tr\na\dd\pi(\dd f_t, \dd f_t) \,.
$$ 
The first term on the right-hand side is $o(t)$ by \eqref{ft-1st-order}(i).  As for the second term,
by \eqref{ft-1st-order}(ii) we have
$$
\na\dd\pi(\dd f_t, \dd f_t) = \na\dd\pi\bigl((\dd f_t)^{\Hh}, (\dd f_t)^{\Hh}\bigr) + o(t) = o(t) \,,
$$ 
since, for any Riemannian submersion, $\na\dd\pi(X,X) =0$ for any horizontal vector $X$  \cite{O'N}.
Hence $\tau(\phi_t) = o(t)$, and by Proposition \ref{prop:first-order}, $v = \pa\phi_t/\pa t|_{t=0}$ is a Jacobi field.

\medskip

(b) \emph{Conversely, let $v$ be a Jacobi field along $\phi$.   We show that there is a unique IHHD $u$ along $f$ such that $\dd\pi \circ u = v$}.
	We establish this by a sequence of lemmas.

Let $C_{\phi}$ be the set of points where $\phi$ fails to be an immersion (i.e., the branch points of $\phi$);  we first of all lift $v$ at points of $S^2 \setminus C_{\phi}$\,.  Let $\{\phi_t\}$ be a smooth one-parameter variation of $\phi$ tangent to $v$.  Then, by Proposition \ref{prop:first-order}, $\phi_t$ is harmonic and isotropic to first order.   For each $x \in S^2 \setminus C_{\phi}$, $\phi_t$ is an immersion at $x$ \emph{for small enough $|t|$}, by which we mean \emph{for $|t| < \ep(x)$ for some $\ep(x) > 0$}; then the twistor lift $f_t(x)$ can be defined as in \eqref{twistor-lift}.  Note that all calculations below will be valid for small enough $|t|$ in this sense.  Set $u(x) = \pa f_t(x)/\pa t|_{t=0}$.  Clearly $\dd\pi \circ u = v$.  We shall show that $u$ is a well-defined infinitesimal horizontal holomorphic deformation of $f$ on $S^2 \setminus C_{\phi}$\,.

By definition, $f_t(x)$ is a positive almost Hermitian structure  on $T_{\phi_t(x)}S^4$.  It thus gives a decomposition
\begin{equation} \label{decomp-type}
T^{\cc}_{\phi_t(x)}S^4 = T'_{\phi_t(x)}S^4 \oplus T''_{\phi_t(x)}S^4 = \ov{P_t(x)} \oplus P_t(x)
\end{equation}
into $(1,0)$- and $(0,1)$- tangent spaces.  In terms of the quadric Grassmannian model $\Tt_{2,5}$ of the twistor space, $f_t(x)$ is represented by the positive isotropic $2$-plane $P_t(x) = T''_{\phi_t(x)}S^4$.

To examine this, let $z$ be a local complex coordinate on an open subset $U$ of $S^2 \setminus C_{\phi}$.  For $x \in U$ set $\psi_t(x) = (\pa\phi_t/\pa \ov{z})'' = (\pa\phi_t/\pa \ov{z})^{\!P_t}$, the $(0,1)$-part, equivalently, $P_t$-component,  of $\pa\phi_t/\pa \ov{z}$ at $x$ with respect to the decomposition \eqref{decomp-type}.  Note that this is non-zero for $t=0$, so that it remains non-zero for small enough $|t|$.  Hence its conjugate $(\pa\phi_t/\pa z)' = (\pa\phi_t/\pa z)^{\!\ov{P_t}}$  is also non-zero.  Note that $P_t(x)$ is the unique positive isotropic $2$-plane in $T^{\cc}_{\phi_t(x)}S^4$ containing $\psi_t(x)$.  Replacing $U$ by a smaller open subset if necessary, let $b_t$ be a non-zero section of $P_t$ over $U$ such that $\{b_t(x), \psi_t(x)\}$ is a basis for $P_t(x)$.  We calculate derivatives of these sections; it is convenient to use $\varPhi_t = i \circ \phi_t$ where $i:S^4 \to \RR^5$ is the standard inclusion.  Note that $\pa\varPhi_t/\pa z = \pa\phi_t/\pa z$ etc.

\begin{lem} \label{lem:psi-t}
Let $x \in U \setminus C_{\phi}$.  Then, for small enough $|t|$,

{\rm (i)} \ $\psi_t = \dfrac{\pa\phi_t}{\pa\ov{z}} + o(t)$\,;

{\rm (ii)} under the orthogonal decomposition\/ $\CC^5 = \ov{P_t(x)} \oplus P_t(x) \oplus \varPhi_t(x)$,
the $P_t$-component $(\pa\psi_t/\pa\bar{z})^{\!P_t}$ of\/ $\pa\psi_t/\pa\bar{z}$ satisfies
$$
\Bigl(\frac{\pa\psi_t}{\pa\bar{z}}\Bigr)^{\!\!\! P_t} = \frac{\pa\psi_t}{\pa\bar{z}} + o(t)
		= \frac{\pa^2\varPhi_t}{\pa\ov{z}^2} + o(t) \,.
$$
\end{lem}

\begin{proof}
(i) This is equivalent to showing that the $(1,0)$-part $(\pa\phi_t/\pa \ov{z})^{\!\ov{P_t}}$ of $\pa\phi_t/\pa \ov{z}$ is $o(t)$.  Now, as above, for small enough $|t|$, the vector $(\pa\phi_t/\pa z)^{\!\ov{P_t}}$ is non-zero, and so is a basis for the one-dim\-en\-sion\-al complex space $(\Im\dd\phi_t)^{\!\ov{P_t}}$.  Hence, $(\pa\phi_t/\pa \ov{z})^{\!\ov{P_t}}$ must be a multiple of it.
Further, by Proposition \ref{prop:isotropy-1st} we have
\begin{multline*}
\Biginn{\Bigl( \frac{\pa\phi_t}{\pa\bar{z}}\Bigr)^{\!\!\!\ov{P_t}},
	\Bigl(\frac{\pa\phi_t}{\pa z}\Bigr)^{\!\!\!\ov{P_t}}}^{\!\hh} =
\Biginn{\Bigl( \frac{\pa\phi_t}{\pa\bar{z}}\Bigr)^{\!\!\!\ov{P_t}},
	\Bigl(\frac{\pa\phi_t}{\pa\bar{z}}\Bigr)^{\!\!\!P_t}}^{\!\cc} \\
	= \frac{1}{2} \Biginn{\frac{\pa\phi_t}{\pa\bar{z}}, \frac{\pa\phi_t}{\pa\bar{z}}}^{\!\cc} + o(t)
	= o(t) \,.
\end{multline*}

This implies that $(\pa\phi_t/\pa \ov{z})^{\!\ov{P_t}}$ is $o(t)$, as required.

(ii) It suffices to show that the components of $\pa\psi_t/\pa\bar{z}$ in $\spn_{\CC}\,{\varPhi_t}$ and in $\ov{P_t}$ are $o(t)$.  By part (i) and Proposition \ref{prop:isotropy-1st}, the component in $\spn_{\CC}\,{\varPhi_t}$ is
\begin{equation} \label{Phi-cpt}
\Biginn{\frac{\pa\psi_t}{\pa\bar{z}}, \varPhi_t}^{\!\cc}
	= \Biginn{\frac{\pa^2\varPhi_t}{\pa\bar{z}^2}, \varPhi_t}^{\!\cc} + o(t) = o(t) \,.
\end{equation}

Next note that, when $t=0$, $\{\psi_t\,, (\pa\psi_t/\pa \ov{z})^{P_t}\}$ equals $\{\pa\varPhi/\pa\ov{z}\,,\pa^2\varPhi/\pa\ov{z}^2\}$; this is a basis for $P_0$ on a dense open subset of $U \setminus C_{\phi}$.  It follows that $\{\psi_t\,, (\pa\psi_t/\pa \ov{z})^{P_t} \}$ is a basis for $P_t$ on that set for small enough $|t|$.  We use this basis to estimate the component in $\ov{P_t}$\,.  

First, since $\psi_t$ is in $P_t$\,,
$$
\Biginn{\Bigl(\frac{\pa\psi_t}{\pa\bar{z}}\Bigr)^{\!\!\!\ov{P_t}}, \psi_t}^{\!\cc}
	= \Biginn{\frac{\pa\psi_t}{\pa\bar{z}}, \psi_t}^{\!\cc}
	= \frac{1}{2} \frac{\pa}{\pa\bar{z}}\inn{\psi_t, \psi_t}^{\cc} = 0\,.
$$

Second, on using part (i), \eqref{Phi-cpt}, and Proposition \ref{prop:isotropy-1st}, we have
$$
\Biginn{\Bigl(\frac{\pa\psi_t}{\pa\bar{z}}\Bigr)^{\!\!\!\ov{P_t}}, \Bigl(\frac{\pa\psi_t}{\pa\bar{z}}\Bigr)^{\!\!\!P_t}}^{\!\cc}
	= \frac{1}{2} \Biginn{\frac{\pa\psi_t}{\pa\bar{z}}, \frac{\pa\psi_t}{\pa\bar{z}}}^{\!\cc} + o(t)
	= \Biginn{\frac{\pa^2\varPhi_t}{\pa\bar{z}^2}, \frac{\pa^2\varPhi_t}{\pa\bar{z}^2}}^{\!\cc} + o(t)
	= o(t) \,.
$$

The last two calculations show that the component of $\pa\psi_t/\pa\bar{z}$ in $\ov{P_t}$ is $o(t)$.   
\end{proof}

We can now examine the derivatives of our basis of $P_t$\,.

\begin{lem} \label{lem:derivs}
\begin{eqnarray*}
{\rm (i)} \ \frac{\pa\psi_t}{\pa z} &\in & \spn_{\CC}\,{\varPhi_t} + o(t) \,;\\
{\rm (ii)} \ \frac{\pa b_t}{\pa z} &\in & P_t + \spn_{\CC}\,{\varPhi_t} + o(t) \,; \\
{\rm (iii)} \ \frac{\pa\psi_t}{\pa\bar{z}} & \in &P_t + o(t) \,; \\
{\rm (iv)} \ \frac{\pa b_t}{\pa\bar{z}} &\in & P_t + o(t) \,.
\end{eqnarray*}
\end{lem}

\begin{proof}
(i) {}From Lemma \ref{lem:psi-t}(i) and \eqref{ha-sphere} we have
$$
\frac{\pa\psi_t}{\pa z} = \frac{\pa^2 \varPhi_t}{\pa z\pa\ov{z}} + o(t) \in \spn_{\CC}\,{\varPhi_t} + o(t) \,. 
$$
(ii) {}From part (i) we have
\begin{eqnarray*}
\Biginn{\frac{\pa b_t}{\pa z} \,, \psi_t}^{\!\cc}
		&= &- \Biginn{b_t \,, \frac{\pa\psi_t}{\pa z}}^{\!\cc} = o(t)\,;
			\quad \text{also} \\
\Biginn{\frac{\pa b_t}{\pa z} \,, b_t}^{\!\cc} &= &\frac{1}{2}\frac{\pa}{\pa z}\biginn{b_t, b_t}^{\!\cc} = 0\,,
\end{eqnarray*}
which shows (ii).

(iii) Immediate from Lemma \ref{lem:psi-t}(ii).

(iv) {}From part (iii) we have
\begin{eqnarray*}
\Biginn{\frac{\pa b_t}{\pa\bar{z}} \,, \psi_t}^{\!\cc} &=& -\Biginn{b_t, \frac{\pa\psi_t}{\pa\bar{z}}}^{\!\cc} = o(t) \,;
		\quad \text{also} \\
\Biginn{\frac{\pa b_t}{\pa \bar{z}} \,, b_t}^{\!\cc} &=& \frac{1}{2}\frac{\pa}{\pa \bar{z}}\biginn{b_t, b_t}^{\!\cc} = 0\,.
\end{eqnarray*}
Further, by Lemma \ref{lem:psi-t}(i),  $\pa\phi_t/\pa\bar{z} \in P_t + o(t)$, so we have
$$
\Biginn{\frac{\pa b_t}{\pa\bar{z}} \,, \phi_t}^{\!\cc} = -\Biginn{b_t, \frac{\pa\phi_t}{\pa\bar{z}}}^{\!\cc} = o(t) \,,
$$
which, together with (ii), shows (iv).
\end{proof}

\begin{lem}
{\rm (i)} The twistor lift $f_t$ is holomorphic to first order and horizontal to first order.

{\rm (ii)} $u = \pa f_t/\pa t|_{t=0}$ is an IHHD defined on $S^2 \setminus C_{\phi}$\,.
\end{lem}

\begin{proof}
(i) Using the quadric Grassmannian model,  the $(1,0)$-part of derivative $\pa f_t/\pa\bar{z}$ is the linear map from $P_t$ to $P_t^{\perp}$ which maps the basis vectors $\psi_t$ and $b_t$ to the $P_t^{\perp}$-components of $\pa\psi_t/\pa\bar{z}$ and $\pa b_t/\pa\bar{z}$ respectively.  By Lemma \ref{lem:derivs}, these components are $o(t)$.  Hence
the $(1,0)$-part of $\pa f_t/\pa\bar{z}$ is $o(t)$, which proves the first assertion. 

Similarly, the $(1,0)$-part of $\pa f_t/\pa z$ is the linear map from $P_t$ to $P_t^{\perp}$ which maps the basis vectors $\psi_t$ and $b_t$ to the $P_t^{\perp}$-components of $\pa\psi_t/\pa z$ and $\pa b_t/\pa z$ respectively.  By Lemma \ref{lem:derivs}, these last quantities lie in $\spn_{\CC}{\phi_t}$ to $o(t)$.  Hence
the $(1,0)$-part of derivative $\pa f_t/\pa z$ lies in $\Ll(P_t, \spn_{\CC}{\phi_t})$ to $o(t)$;  this is the $(1,0)$ horizontal space by \eqref{HV-P}.

(ii) This follows from (i) by Proposition \ref{prop:inf-hol} and Definition \ref{def:inf-HH}.
\end{proof}

\begin{lem}
The vector field $u$ can be extended to an IHHD on $S^2$.
\end{lem}

\begin{proof}
With respect to the pull-back of the decomposition \eqref{HV-Si} write $u = u^{V} + u^{H}$; then we have the corresponding decomposition \eqref{HV-Si'} of $(1,0)$-parts: $u' = u^{V'} + u^{H'}$.
Note that $u^{H}$ is the horizontal lift of $v$, and so both it and its $(1,0)$-part $u^{H'}$ can be extended to real-analytic vector fields on $S^2$.

Noting that the K\"ahler nature of $\CP^3$ means that covariant derivatives respect the type decomposition, holomorphicity of $u'$ tells us that, with respect to any local complex coordinate $z$ on $S^2$,
\begin{eqnarray*}
0 &= &\na^f_{\pa/\pa\bar{z}}u' \\
	&= &V'(\na^f_{\pa/\pa\bar{z}}u^{V'}) + V'(\na^f_{\pa/\pa\bar{z}}u^{H'})
		+ H'(\na^f_{\pa/\pa\bar{z}}u^{V'})  + H'(\na^f_{\pa/\pa\bar{z}}u^{H'}) \\
	&= &\na^{V'}_{\pa/\pa\bar{z}}u^{V'} + 0
		+  A''_{V'}u^{V'} + \na^{H'}_{\pa/\pa\bar{z}}u^{H'},
\end{eqnarray*}
where $\na^{V'}$ and $\na^{H'}$ denote the induced connections on $V'$ and $H'$; the second term is zero by holomorphicity of $H'$ (Proposition \ref{prop:hor-sbb2}(i)).  
Taking components in $V'$ and $H'$ we obtain 
\begin{equation} \label{hol-u}
\text{(i) } \na^{V'}_{\pa/\pa\bar{z}}u^{V'} = 0 \text{ and }
		\text{(ii) }  A''_{V'}u^{V'} = -\na^{H'}_{\pa/\pa\bar{z}}u^{H'}.
\end{equation}
The first equation says that $u^{V'}$ is a holomorphic section of $V'$ with respect to its Koszul--Malgrange holomorphic structure.  Hence, at any point of $C_{\phi}$, it has a removable singularity, a pole or an isolated essential singularity.

However, by Proposition \ref{prop:hor-sbb2}(ii), $A''_{V'}$ is antiholomorphic, so near a point of $C_{\phi}$, in a local complex coordinate $z$ centred on that point, it is of the form $\bar{z}^k$ times a real-analytic non-zero map for some $k \in \{0,1,2,\ldots\}$.   It follows that $\bar{z}^k u^{V'}$ is smooth, which is not possible if $u^{V'}$ has a pole or isolated essential singularity.

Hence $u^{V'}$ has a removable singularity at each point of $C_{\phi}$\,, and so can be extended smoothly to $S^2$.  It follows that $u'$, and so also $u$, can be extended smoothly to $S^2$.
\end{proof}

\begin{lem} \label{lem:unique-vf}
There is at most one holomorphic vector field $u$ with $\dd\pi \circ u = v$.
\end{lem}

\begin{proof}
As in the last proof, the horizontal component of $u$ is uniquely determined as the horizontal lift of
$v$.  The vertical component is determined away from zeros of $A''_{V'}$ by \eqref{hol-u}(ii); since the complement of the zeros is a dense set, by continuity this component is also unique.
\end{proof}

This completes the proof of Theorem \ref{th:Jacobi-lift}.

\end{proof}

We now turn to the \emph{integrability} (Definition \ref{def:integrable}) of the corresponding infinitesimal deformations.

\begin{prop} \label{prop:int-full}
Let $\phi:S^2 \to S^4$ be a full harmonic map of positive spin and let $f:S^2 \to \Si^+ = \CP^3$ be its twistor lift. 
Let $u$ be an IHHD along $f$ and $v = \dd\pi \circ u$ the corresponding Jacobi field along $\phi$. Then the following are equivalent:

{\rm (i)} $v$ is integrable;

{\rm (ii)} $u$ is integrable by horizontal holomorphic maps. 
\end{prop}

\begin{proof}
Suppose that (i) holds.  Then there is a one-parameter family of harmonic maps $\phi_t:S^2 \to S^4$ with $\phi_0 = \phi$ and $\pa\phi_t/\pa t|_{t=0} = v$.  For small enough $|t|$, $\phi_t$ has positive spin and so its twistor lift is a horizontal holomorphic map $f_t:S^2 \to \CP^3$.  Clearly, $f_0 = f$ and $\pa f_t/\pa t|_{t=0} = u$, hence (ii) holds.  The converse is similar.
\end{proof}

\begin{defn} \label{def:regular-ha}
Let $\phi \in \Harm_d^{\ff}(S^2,S^4)$.  Recall that one of $\pm\phi$ has positive spin and so has a twistor lift $f \in \HHol_d^{\ff}(S^2,\CP^3)$.  Say that \emph{$\phi$ is regular} if $f$ is regular in the sense of Definition \ref{def:regular-ho}.
\end{defn}

By Theorem \ref{th:smooth}, a regular point is a smooth point of $\Harm^{\ff}_d(S^2,S^4)$.  By Theorem \ref{th:Jacobi-lift} and Proposition \ref{prop:int-full}, we obtain the main result of this section.

\begin{thm} \label{th:full-int}
Let $d \in \{3,4,\ldots\}$ and let $\phi \in \Harm^{\ff}_d(S^2,S^4)$.  Then the following are equivalent:

{\rm (i)}  $\phi$ is a regular point;

{\rm (ii)}  the dimension of the space of Jacobi fields at $\phi$ is $2d+4$;

{\rm (iii)} all Jacobi fields along $\phi$ are integrable.
\end{thm}

\begin{proof}  On composing with an orientation-reversing isometry, if necessary, we can assume that $\phi$ is of positive spin and thus given as the projection $\pi \circ f$ of a full horizontal holomorphic map.  The result then follows by combining Propositions  \ref{prop:smooth} and \ref{prop:int-full}.
\end{proof}

J.\ Bolton and L.~M.\ Woodward \cite{BoWo-space} show that  all $f \in \HHol_d^{\ff}(S^2,\CP^3)$ are regular; in particular, $\Harm_d^{\ff}(S^2,S^4)$ is a smooth manifold\footnote{Bolton and L. Fern\'andez inform us that they have also proved this for $d=6$.} for $d = 3,4,5$.

\begin{cor} \label{cor:full-int}
All Jacobi fields along a full harmonic map of twistor degree $\leq 5$ are integrable.
\end{cor}

\section{The non-full case} \label{sec:nonfull}

\subsection{Infinitesimal deformations of non-full horizontal holomorphic maps}

Recall that a map $S^2 \to \CP^3$ is non-full if and and only if its image lies in a projective subspace $\CP^2 \subset \CP^3$.
We study the space of non-full horizontal holomorphic maps, $\HHol^{\nf}(S^2,\CP^3)$.

It turns out that any such map has image in a $\CP^1$; this must be horizontal in the sense that its tangent spaces are in the horizontal subbundle $\Hh$ of \eqref{HV-CP3}.  In fact, let $\CP^1_0 = \{[z_0,0,z_2,0] : [z_0,z_2] \in \CP^1\}$; note that this is a horizontal $\CP^1$.  Recalling that $\Sp(2)$ acts on $\CP^3$ preserving horizontality, we have 

\begin{lem} \label{lem:CP0^1} Let $f:S^2 \to \CP^3$ be a non-full horizontal holomorphic map.  Then there is $A \in \Sp(2)$ such that $A \circ f$ has its image in $\CP^1_0$.
\end{lem}

\begin{proof} First, we can choose $A_1 \in \Sp(2)$ such that $\wt{f} = A_1 \circ f$ has the form
$\wt{f} = [F_0,0,F_2,F_3]$. Then the horizontality condition \eqref{horizontal} gives
$$
F_2 F_3' - F_3 F_2' = 0
$$
which implies that $F_3/F_2$ is constant.  We next choose $A_2 \in \Sp(2)$ such that, after  composing with $A_2$, that constant is zero.  Then $A_2 \circ A_1\circ f$ has image in $\CP^1_0$.
\end{proof}

The proof shows that $\Sp(2)$ acts transitively on the set of horizontal $\CP^1$s; clearly the isotropy group is a copy of $\U(2)$, hence the set of horizontal $\CP^1$s can be identified with the complex homogeneous space $\Sp(2)/\U(2)$ of dimension $3$.

Fix $d \in \{1,2,\ldots\}$.  It is clear that $\Hol_d(S^2,\CP^1_0)$ is a complex submanifold of $\Hol_d(S^2,\CP^3)$ of dimension $2d+1$.  We deduce the following.

\begin{prop} \label{prop:space-nonfull}
The space $\HHol_d^{\nf}(S^2,\CP^3)$ is a complex submanifold of\/ $\Hol_d(S^2,\CP^3)$ of dimension $2d+4$.
\end{prop}

\begin{proof}
The map $\HHol_d^{\nf}(S^2,\CP^3) \to \Sp(2)/\U(2)$ given by mapping $f$ to its image --- a horizontal $\CP^1$ ---  is easily seen to be a locally trivial fibre bundle with fibres biholomorphic to $\Hol_d(S^2,\CP^1)$.  So $\HHol_d^{\nf}(S^2,\CP^3)$ is a complex manifold of dimension $(2d+1)+3 = 2d+4$.  Since its topology and complex structure are those induced from $\Hol_d(S^2,\CP^3)$, it is actually  a complex \emph{sub}manifold of $\Hol_d(S^2,\CP^3)$.
\end{proof}

Now, $\HHol_d(S^2,\CP^3)$ is the disjoint union of $\HHol_d^{\ff}(S^2,\CP^3)$ and $\HHol_d^{\nf}(S^2,\CP^3)$.  Although $\HHol_d^{\nf}(S^2,\CP^3)$ is closed in $\HHol_d(S^2,\CP^3)$, the subspace $\HHol_d^{\ff}(S^2,\CP^3)$ is not; this motivates the following definition.

\begin{defn} \label{def:collapse}
Say that $f \in \HHol_d^{\nf}(S^2,\CP^3)$ is a \emph{collapse point} if it is in $\ov{\HHol_d^{\ff}(S^2,\CP^3)} \cap \HHol_d^{\nf}(S^2,\CP^3)$.
\end{defn}

Thus a non-full horizontal holomorphic map $f$ is a collapse point if it occurs as the limit of a \emph{sequence} of full horizontal holomorphic maps; since $\HHol_d(S^2,\CP^3)$ is an algebraic variety, this is equivalent to the existence of a one-parameter \emph{family} of full horizontally holomorphic maps which `collapses' to $f$.

The set of collapse points $f$ forms a proper algebraic subvariety of $\HHol_d^{\nf}(S^2,\CP^3)$.
For $d=1$ or $2$, $\HHol_d^{\ff}(S^2,\CP^3)$ is empty so that the set of collapse points is empty. The following examples show that, for each $d \geq 3$, the set of collapse points is non-empty; see \S \ref{sec:collapse} for more information on that set.   

\begin{exm} \label{ex:d3} Let $f_t(z) = [z,3tz^2,z^3+1,t]$.  For $t \neq 0$, this defines a full horizontal holomorphic map $f_t:S^2 \to \CP^3$ of degree $3$.  As $t \to 0$, it approaches the non-full horizontal holomorphic map $f_0:S^2 \to \CP^3$ of the same degree given by $f_0(z) = [z,0,z^3+1,0]$; by definition, the map $f_0$ is a collapse point.
\end{exm}

This is the only example for $d=3$ up to fractional linear transformations, see \S \ref{sec:collapse}.   For higher values of $d$, there is a greater variety of examples; we give one family (equivalent to the previous one when $d=3$).

\begin{exm} \label{ex:d-arb} \cite{Ej-min,EjKo-extra} Let $d \geq 3$. For any $\alpha \neq 0$ and any positive integers $\ell, k$ with $\ell \neq k$ and $\ell + k = d$, let
$$
f_t(z) = \left[z^k - \al\,,\, t\,\frac{\ell+k}{\ell-k}z^{\ell}\,,\, z^{\ell}\Big(z^k - \al\frac{\ell+k}{\ell-k} \Big)\,,\, t \right]\!.
$$
For $t \neq 0$, this defines a full horizontal holomorphic map $f_t:S^2 \to \CP^3$ of degree $d$.
As $t \to 0$, it approaches the non-full horizontal holomorphic map $f_0:S^2 \to \CP^3$ of the same degree given by 
$f_0(z) =  \bigl[z^k - \al\,, 0, z^{\ell}\bigl(z^k - \al\frac{\ell+k}{\ell-k}\bigr), 0\bigr]$; 
by definition, the map $f_0$ is a collapse point.
\end{exm} 

Let $f \in \HHol_d^{\nf}(S^2,\CP^3)$; without loss of generality, by Lemma \ref{lem:CP0^1}, we can assume that $f \in \Hol_d(S^2,\CP^1_0)$.  Write $f = [F]$ where $F = (F_0, 0,F_2,0)$ is a holomorphic lift.  Let $u$ be a holomorphic vector field along $f$.   As in \S \ref{subsec:inf-hol}, $u = \dd p_F(\U)$ for some holomorphic $\U = (\U_0,\U_1,\U_2,\U_3)$.    Let $u = u^T + u^{\perp}$ be the decomposition of $u$ into components tangential and normal to $\CP^1_0$.  Then $u^T = \dd p_F(\U^T)$ and $u^{\perp} =  \dd p_F(\U^{\perp})$ where $\U^T = (\U_0, 0, \U_2,0)$ and $\U^{\perp} = (0,\U_1, 0, \U_3)$, respectively.

Now, from \eqref{inf-hor},  $u$ is an infinitesimal horizontal deformation if and only if
\begin{equation} \label{Q-NF0}
\dd Q_F(\U) =0
\end{equation}
where
\begin{equation} \label{Q-NF}
\dd Q_F(\U) = - F_0 \U_1' + F_0' \U_1 - F_2 \U_3' + F_2' \U_3\,.
\end{equation}
Note that this condition is well defined and only involves the normal part $u^{\perp}$ of $u$.

On taking $F$ and $\U$ to be polynomial of degree $\leq d$, the differential $\dd Q_F$  can be thought of as a linear map from $\CC[z]_d^2$ to $\CC[z]_{2d-2}$, so has kernel of dimension
$\dim\ker\dd Q_F = 2d+2 - \dim\image\;\dd Q_F \geq 3$.
There are three obvious solutions to this: $(\U_1,\U_3) = (F_0,0)$, $(0,F_2)$ and $(F_2,F_0)$.  It is clear that these span the three-dimensional subspace of $T_f\HHol_d^{\nf}(S^2,\CP^3)$ which is tangent to the action of $\Sp(2)$; thus they are all integrable.
Say that $u$ is an \emph{extra IHHD} if $u^{\perp}$ is not in this three-dimensional space; equivalently, $u$ is not tangent to $\HHol_d^{\nf}(S^2,\CP^3)$.  

\begin{prop} \label{prop:jn-pt}
Let $f:S^2 \to \CP^3$ be a non-full horizontal holomorphic map.   Then the following are equivalent.

{\rm (i)} $f$ is a non-smooth point of\/ $\HHol_d(S^2,\CP^3)$;

{\rm (ii)} $f$ is a collapse point;

{\rm (iii)} $f$ has an extra IHHD;

{\rm (iv)} $f$ has an extra IHHD which is not integrable by horizontal holomorphic maps;

{\rm (v)} $f$ has an extra IHHD which is integrable by horizontal holomorphic maps.
\end{prop}

\begin{proof}
The equivalence of (i) and (ii) holds since, by Proposition \ref{prop:space-nonfull}, $\HHol_d^{\nf}(S^2,\CP^3)$ is a submanifold of $\HHol_d(S^2,\CP^3)$ at $f$ if and only if $f$ is not the limit of maps in $\HHol_d^{\ff}(S^2,\CP^3)$.

That (i) implies (iv) follows from Proposition  \ref{prop:smooth}.  That (iv) or (v) implies (iii) is trivial.

To show that (iii) implies (v) and (ii),  suppose that $u = \dd p_F(\U)$ is an extra IHHD along $f$.   Then, by \eqref{Q-NF0}, $f_t = [F_t +t\U^{\perp}]$ is actually horizontal for all $t$.  We claim that there exists $\ep >0$ such that $f_t$ is full for all non-zero $t$ with $|t| < \ep$.  If not, there would be a sequence of non-zero values of $t$ tending to $0$ such that $f_t$ is not full. But this would imply that $u$ is tangent to $\Hol_d^{\nf}(S^2,\CP^3)$, which it is not.   Thus $f$ is a collapse point, and $u^{\perp}$ is an extra IHHD which is integrable.
\end{proof}

Recalling Definition \ref{def:regular-ho}, we deduce

\begin{cor} \label{cor:smooth-reg}  A non-full horizontal holomorphic map $f:S^2 \to \CP^3$ is a smooth point of\/ $\HHol_d(S^2,\CP^3)$ if and only if it is a regular point. 
\qed
\end{cor}

It follows from Proposition \ref{prop:jn-pt} that we can determine whether a point $f$ is a collapse point by finding all solutions $(\U_1,\U_3)$ to the linear equations \eqref{Q-NF0} with $\U_1$ and $\U_3$ polynomials of degree $\leq d$.  If the space of all solutions is more than $3$-dimensional, then there are extra IHHDs and $f$ is a collapse point; otherwise it is not.  It is an exercise in linear algebra to see, for example, that  $f:S^2 \to S^2$ defined by $z \mapsto z^d$ is \emph{not} a collapse point for any $d$; this also follows from Lemma \ref{lem:Q-subm}.

\subsection{Infinitesimal deformations of non-full harmonic maps} \label{subsec:inf-ha}
Fix $d \in \{1,2,\ldots\}$.  By Theorem \ref{th:smooth},  the twistor projection \eqref{Pi-d-imm} maps $\HHol_d^{\nf}(S^2,\CP^3)$ diffeomorphically onto the space $\Harm_d^{\nf}(S^2,S^4)$.  Since the former is a complex-analytic submanifold of $\Hol_d(S^2,\CP^3)$, the latter is a real-analytic submanifold of $R_d$.  We now look at maps in $\Harm_d^{\nf}(S^2,S^4)$ in more detail.

Let  $S^2_0 = \{(x_1,x_2,x_3,0,0) \in S^4\}$; this is a totally geodesic $S^2$ in $S^4$.  {}From \cite{Ca2} we have

\begin{prop} 
Any non-full harmonic map $\phi:S^2 \to S^4$ is the composition of a weakly conformal map with image a totally geodesic $S^2$ and the inclusion map $i:S^2 \hookrightarrow S^4$.    In fact, up to isometries of $S^4$, we may assume that $\phi = i \circ \phi_0$ for some holomorphic map $\phi_0:S^2 \to S^2_0$. 
\qed \end{prop}
 
Now the twistor projection \eqref{twistor-CP3} maps $\CP^1_0$ onto $S^2_0$  isometrically; in fact, it is just the standard identification \eqref{CP1S2} of $\CP^1$ with $S^2$. Let the twistor lift of $\phi:S^2 \to S^2_0 \hookrightarrow S^4$ be $f:S^2 \to \CP^3$.  Then $f = i \circ f_0$ where $i:\CP^1_0 \hookrightarrow \CP^3$ is the inclusion map and $f_0:S^2 \to \CP^1_0$ is holomorphic.  The twistor projection identifies $f_0$ with $\phi_0$, so that \emph{the non-full map $\phi$ and its twistor lift $f$ may be considered to be the same map}.

The Jacobi equation simplifies considerably for non-full maps, as follows.

\begin{lem}
A vector field along a non-constant non-full harmonic map $\phi:S^2 \to S^4$ is Jacobi if and only if its tangential and normal parts are Jacobi.  Further, the tangential part is Jacobi if and only if it is conformal.  In fact, if\/ $\phi$ is the composition of a holomorphic map $S^2 \to S^2_0$ and the inclusion map, then the tangential part is holomophic.
\end{lem}

\begin{proof}
That the image of $\phi$ is totally geodesic implies that the Jacobi equation \eqref{Jacobi} splits in tangential and normal parts.  Conformality and holomorphicity follow easily from Proposition \ref{prop:hol}.
\end{proof}

\begin{defn} A Jacobi field $v$ along a non-full harmonic map $\phi \in \Harm_d^{\nf}(S^2,S^4)$ is called \emph{extra} if it is not tangent to $\Harm_d^{\nf}(S^2,S^4)$.   
\end{defn}

Note that (i) \emph{$v$ is extra if and only if its normal part $v^{\perp}$ is extra}; hence, as for the holomorphic case in the previous section, it suffices to consider normal vector fields in the sequel;
 (ii) a normal vector field is tangent to $\Harm_d^{\nf}(S^2,S^4)$ (i.e., is \emph{not} extra) if and only if it is tangent to a rigid motion of $S^2$ in $S^4$.

 As in the case of holomorphic maps, we make the following definition.

\begin{defn} \label{def:collapse-ha}
Say that $\phi \in \Harm_d^{\nf}(S^2,S^4)$ is a \emph{collapse point} if it is in $\ov{\Harm_d^{\ff}(S^2,S^4)} \cap \Harm_d^{\nf}(S^2,S^4)$.
\end{defn}

To use the twistor construction, we must restrict to full harmonic maps \emph{of positive spin}; the next lemma shows that this is not a problem.

\begin{lem} \label{lem:jn-pts-ha}
For any $d \geq 3$, the sets \ $\ov{\Harm_d^{\ff}(S^2,S^4)} \cap \Harm_d^{\nf}(S^2,S^4)$,
\ $\ov{\Harm_d^+(S^2,S^4)} \cap \Harm_d^{\nf}(S^2,S^4)$ \ and
\ $\ov{\Harm_d^-(S^2,S^4)} \cap \Harm_d^{\nf}(S^2,S^4)$ are all equal.

Hence $\phi \in \Harm_d^{\nf}(S^2,S^4)$ is a collapse point if and only if its twistor lift $f \in \HHol_d^{\nf}(S^2,\CP^3)$ is a collapse point.
\end{lem}

\begin{proof} Suppose that $\phi \in \ov{\Harm_d^+(S^2,S^4)} \cap \Harm_d^{\nf}(S^2,S^4)$; then $\phi$ is the limit of a sequence $\{\phi_n\}$ in $\Harm_d^+(S^2,S^4)$. Without loss of generality, assume that $\phi \in \Harm_d(S^2,S^2_0)$.  The involution $(x_1,x_2,x_3,x_4,x_5) \mapsto (x_1,x_2,x_3,x_4,-x_5)$ of $S^4$ fixes $S^2_0$.  Composing a harmonic map $S^2 \to S^4$ with this involution changes its spin, and so maps the sequence $\{\phi_n\}$ to a sequence in $\Harm_d^-(S^2,S^4)$, which also tends to $\phi$.
\end{proof}

We can thus find collapse points in $\Harm_d^{\nf}(S^2,S^4)$ for each $d \geq 3$, as follows.

\begin{exm} \label{exm:collapse-ha}
Applying the twistor projection \eqref{twistor-CP3} to Examples \ref{ex:d3} and \ref{ex:d-arb} gives:

(i) a family of full harmonic maps $\phi_t:S^2 \to S^4$ of degree $3$ which collapses to a non-full harmonic map $\phi_0:S^2 \to S^2_0$ of the same degree given by $\phi_0(z) = (z^3+1)/z$;  

(ii) for any $d \geq 3$, a family of full harmonic maps $\phi_t:S^2 \to S^4$ of degree $d$ which collapses to a non-full harmonic map $\phi_0:S^2 \to S^2_0$ of the same degree given by
$\phi_0(z) = \bigl\{z^k - \al \bigr\}\big/\bigl\{z^{\ell}\big(z^k - \al\frac{\ell+k}{\ell-k}\bigr)\!\bigr\}$.
\end{exm}

We try to translate the results in the last section into results on non-full harmonic maps using the map $\Pi_d$.  But there is a problem as the image of $\Pi_d$ is not the whole of $\Harm_d(S^2,S^4)$ but is only the subspace $\Harm_d^{\geq 0}(S^2,S^4)$ of harmonic maps  of non-negative spin.  Further, the proof that Jacobi fields along full maps can be lifted breaks down for non-full maps.  However we have

\begin{prop}
Let $\phi:S^2 \to S^4$ be a non-full harmonic map and let $f:S^2 \to \Si^+ = \CP^3$ be its twistor lift.

The mapping $\dd\Pi_f: u \mapsto v = \dd\pi \circ u$ defines an injective map from the space of infinitesimal horizontal holomorphic deformations $u$ along $f$ to the space of Jacobi fields $v$ along $\phi$; further $u$ is normal if and only if $v$ is.
\end{prop}

\begin{proof}
The proof of part (a) of Theorem \ref{th:Jacobi-lift} still applies in the non-full case; injectivity follows as in Lemma \ref{lem:unique-vf}, and preservation of normality is clear since $\pi$ is a Riemannian submersion.
\end{proof}

We cannot analyse Jacobi fields by lifting them as we did for the full case; however, we have a new tool, developed in \cite{MoRo,EjKo-extra}. 

Let $\phi \in \Harm^{\nf}_d(S^2,S^4)$.   As above,
without loss of generality we may assume that $\phi = \pi \circ f$ for some $f \in \Hol_d(S^2,\CP^1_0)$. 

Let $\{e_1,e_2,e_3,e_4,e_5\}$ be the standard basis for $\RR^5$ so that the normal space of $S^2_0$ in $S^4$ is spanned by $e_4$ and $e_5$.   The following is easy to deduce from \eqref{Jacobi}, or by differentiating \eqref{ha-sphere0}.

\begin{lem}
A normal vector field $v = v_1 e_4 + v_2 e_5$ along $\phi$ is Jacobi if and only if each component satisfies the \emph{generalized eigenvalue (Schr\"odinger) equation}{\rm :}
\begin{equation} \label{e-value}
\Delta v_i = |\dd\phi|^2 v_i \qquad (i=1,2).
\end{equation}
\qed \end{lem}

Write $\varPhi = i \circ \phi = (\varPhi_1,\varPhi_2,\varPhi_3)$ where $i:S^2 \to \RR^3$ is the standard inclusion.  Equation \eqref{ha-sphere} implies that each $\varPhi_i$ is a solution to \eqref{e-value}; hence any linear combination of the $\varPhi_i$ is also a solution, giving a three-dimensional space of `trivial' eigenfunctions.   We call any other solution an \emph{extra eigenfunction}.   The following is clear.

\begin{lem} \label{lem:extra}
A normal Jacobi field $v = v_1 e_4 + v_2 e_5$ is an extra Jacobi field if and only if only if at least one of the $v_i$ is an extra eigenfunction.
\qed \end{lem}

The next result uses an interesting correspondence between solutions to \eqref{e-value} and minimal branched immersions, see \cite{MoRo,EjKo-extra,Ej-min, Ej-bound,Ko}.

\begin{lem} \label{lem:Kotani}
Let $\phi:S^2 \to S^2_0$ be a holomorphic map and let $v_1$ be  an extra eigenfunction.  Then there is another extra eigenfunction $v_2$ such that $v_1 e_4 + v_2 e_5$ is an integrable extra Jacobi field.
\end{lem}

\begin{proof} As in \cite{MoRo,EjKo-extra}, there is a non-constant weakly conformal harmonic map (i.e., minimal branched immersion) $X:S^2 \setminus C_{\phi} \to \RR^3$, unique up to an additive constant,  with (i) $\pa X/\pa z \in$ the subbundle spanned by $\pa\phi/\pa{\ov{z}}$, (ii) $\inn{X,\varPhi} = v_1$ and (iii) $X \to \infty$ as we approach $C_{\phi}$\,.  In fact, $X = v_1 \varPhi + \grad v_1$ is such a map.   The components $\pa X/\pa z$ have poles at points of $C_{\phi}$ of order at least $2$.   Let $Y$ be the harmonic conjugate of $X$.  Since $\pa X/\pa z$ has no residues, $Y$ is well defined on $S^2 \setminus C_{\phi}$.  Set $v_2 = \inn{Y,\varPhi}$.  Then $v_2$ extends smoothly to the whole of $S^2$ and is another extra eigenfunction.  As in \cite{Ej-bound,EjKo-extra,Ko}, we may construct a deformation of $\phi$ tangent to $v_1 e_4 + v_2 e_5$.
\end{proof}

\begin{rem} \label{rem:Kotani}  Thinking of the space of extra eigenfunctions as the quotient of the space of all eigenfunctions by the three-dimensional space of trivial ones, it follows from the lemma that the dimension of the space of extra eigenfunctions is even.  M.\ Kotani \cite{Ko} actually shows that, if $\phi$ has $r$ pairs of extra eigenfunctions, then there is a deformation $\phi_t$ of $\phi$ with $\phi_t:S^2 \to S^{2+2r}$ full harmonic for small non-zero $|t|$.
Now \cite{Ca1} full harmonic maps from $S^2$ to $S^{2+2r}$ exist if and only if the twistor degree $d \geq (r+1)(r+2)/2$; on solving this we obtain the bound $r \leq \frac{1}{2}(\sqrt{8d+1}-3)$.  See \cite{GuOh} for more information on deformations,  and  \cite{Fer-Sn} for a description of the space of harmonic maps from $S^2$ to $S^m$.
\end{rem}

\begin{thm} \label{th:jn-S4}
Let $\phi:S^2 \to S^4$ be a non-full harmonic map.  Denote its twistor degree by $d$.   Then the following are equivalent.

{\rm (i)} $\phi$ is a non-smooth point of\/ $\Harm_d(S^2,S^4)$;

{\rm (ii)} $\phi$ is a collapse point;

{\rm (iii)} $\phi$ has an extra Jacobi field;

{\rm (iv)} $\phi$ has an extra Jacobi field which is integrable;

{\rm (v)} $\phi$ has an extra Jacobi field which is not integrable.
\end{thm}

\begin{proof}
(i) $\Longleftrightarrow$ (ii).  This follows from the fact that $\Harm_d^{\nf}(S^2,S^4)$ is a manifold, cf.\ Proposition \ref{prop:jn-pt}.

(ii) $\implies$ (iii).  Let $\phi \in \Harm_d^{\nf}(S^2,S^4)$.  Then $\phi = \pi \circ f$ where  $f \in \HHol_d^{\nf}(S^2,\CP^3)$.   By Lemma \ref{lem:jn-pts-ha}, $f$ is a collapse point, so by Proposition \ref{prop:jn-pt}, there is an extra IHHD.  This descends to an extra Jacobi field.

(iii) $\implies$ (iv). Given an extra Jacobi field, Lemma \ref{lem:Kotani} provides an extra integrable Jacobi field.

(iv) $\implies$ (ii).  Let $v$ be an extra integrable Jacobi field.  Choose a one-parameter family $\{\phi_t\}$ of harmonic maps tangent to $v$.  As in the proof of Proposition \ref{prop:jn-pt}, we can show that $\phi_t$ is full for $|t|$ non-zero and small enough, so that $\phi$ is a collapse point.

(ii) $\implies$ (v).   The space of Jacobi fields tangent to $\Harm_d^{\nf}(S^2,S^4)$ has dimension $2d+4$.  Since there is an extra Jacobi field, the space of all Jacobi fields along $\phi$ is strictly bigger than this.   If they were all integrable, then by the result of D.\ Adams and L.\ Simon (Proposition \ref{prop:AdamsSimon}), close to $f$ the space $\Harm_d(S^2,S^4)$ would be a manifold of dimension greater than $2d+4$, contradicting Verdier's result that it is a complex algebraic variety of pure dimension $2d+4$.
\end{proof}

As in the holomorphic case (Corollary \ref{cor:smooth-reg}), it follows that \emph{a non-full horizontal harmonic map $f:S^2 \to S^4$ is a smooth point of\/ $\Harm_d(S^2,S^4)$ if and only if it is a regular point} (Definition \ref{def:regular-ha}).

Another consequence of Theorem \ref{th:jn-S4} is

\begin{cor}
Let $n = 3$ or $4$.  Then, for each $d \geq 3$, there are non-full harmonic maps from $S^2$ to $S^n$ of degree $d$ which admit non-integrable Jacobi fields.
\end{cor}

\begin{proof}   The case $n=4$ follows from the last theorem and the fact that the set of collapse points is non-empty. 

For the case $n =3$, let $\phi:S^2 \to S^2_0 \subset S^4$ be a holomorphic map with an extra integrable Jacobi field $v = v_1 e_4 + v_2 e_5$.  By Lemma \ref{lem:extra}, one of the $v_i$ must be an extra eigenfunction, say $v_1$.   Now regard $\phi$ as a harmonic map into $S^3 = \{(x_1,x_2,x_3,x_4,x_5):x_5 = 0 \}$. Then $v_1 e_4$ is a Jacobi field along it which cannot be integrable, since $\Harm_d(S^2,S^3) \subset \Harm^{\nf}_d(S^2,S^4)$, and being extra, $v_1 e_4$ is not tangent to this space.

\end{proof}

\section{Area and nullity} \label{sec:area-nullity}
The \emph{nullity} of (the energy) of a harmonic map is the \emph{real} dimension of the space of Jacobi fields along it.  Combining Theorems \ref{th:full-int} and  \ref{th:jn-S4} we obtain a result valid in both the full and non-full cases.

\begin{thm} \label{th:all-int}
Let $\phi:S^2 \to S^4$ be a harmonic map of twistor degree $d$.  Then the nullity of\/ $\phi$ is greater than or equal to $4d+8$ with equality if and only if\/ $\phi$ is a regular point of the algebraic variety $\Harm_d(S^2,S^4)$.
\end{thm}

Recalling the result of Bolton and Woodward \cite{BoWo-space} cited at the end of \S \ref{sec:main}, we deduce

\begin{cor} \label{cor:d345} The nullity of a full harmonic map $\phi:S^2 \to S^4$ of degree $\leq 5$ is exactly $4d+8$.
\end{cor}

As in \cite[\S 6]{LeWo2}, we can consider instead the second variation of the \emph{area}.  This is unaffacted by tangential conformal fields.  In fact, results of N. Ejiri and M. Micallef \cite{EjMi} imply that,
\emph{for any non-constant harmonic map from the $2$-sphere, the map $v \mapsto$ the normal component of\/ $v$
is a surjective linear map from the space of Jacobi fields for the energy to the space of Jacobi fields for the area, with kernel the tangential conformal fields}.

S.\ Montiel and F.\ Urbano \cite[Corollary 7]{MoUr} show that the nullity of the (second variation of the) area of a (full or non-full) minimal immersion of\/ $S^2$ in $S^4$ of twistor degree $d$ is exactly $4d+2$.   Now the tangential conformal fields form a space of (real) dimension $6$.  Hence the nullity of the energy is precisely $4d+8$; on using Theorem \ref{th:all-int} we deduce that \emph{any immersive harmonic map is a regular, and so a smooth, point of\/ $\Harm_d(S^2,S^4)$}.
 
\section{The set of collapse points} \label{sec:collapse}

\subsection{Generalities}

Fix $d \in \{1,2, \ldots\}$.  We study the set of maps in $\Harm^{\nf}_d(S^2,S^4)$ which are collapse points in the sense of Definition \ref{def:collapse-ha}.  It suffices to consider the set $C_d$ of collapse points in $\Harm_d(S^2,S^2_0)$ $\cong$ $\Harm_d(S^2,S^2)$.  Further, via the twistor construction, we may identify this with the set of collapse points in the $(2d+1)$-dimensional complex manifold $\Hol_d(S^2,\CP^1_0) \cong \Hol_d(S^2,\CP^1)$; indeed, as remarked in \S \ref{subsec:inf-ha}, since the twistor projection \eqref{twistor-CP3} maps $\CP^1_0$ onto $S^2_0$ isometrically, we may consider $f \in \Hol_d(S^2,\CP^1)$ and $\phi = \pi \circ f \in \Harm_d(S^2,S^2)$ to be the same map.

{}From Example \ref{exm:collapse-ha},  $C_d$ is non-empty if and only if $d \geq 3$, in which case it is a non-empty proper algebraic subvariety of $\Hol_d(S^2,\CP^1)$.   Further, by Theorem \ref{th:jn-S4},  $C_d$ is precisely the set of $f \in \Hol_d(S^2,\CP^1)$ which admit extra eigenfunctions.  This last subject was studied in \cite{Ej-min,Ej-bound,Ko,EjKo-extra} and some results on $C_d$ were obtained; we give some new results in \S\S \ref{subsec:invt}, \ref{subsec:branch}.

The space $\Hol_d(S^2,\CP^1) = \Hol_d(S^2,\CP^1_0)$ consists of projective classes of pairs $(F_0,F_2) = (F_0,0,F_2,0)$ of coprime polynomials of maximum degree exactly $d$.  Consider the map which gives the \emph{Wronskian} of $F_0$ and $F_2$:
\begin{equation} \label{Wtilde0}
\wt{W}:\CC[z]_d^2 \to  \CC[z]_{2d-2}\,, \quad  (F_0,F_2) \mapsto -\{F_0,F_2\} = F_0 F'_2 - F_2 F'_0 \,.
\end{equation}
Since this is bilinear and antisymmetric, it induces a linear map
\begin{equation} \label{Wtilde}
\wt{W}:  \wedge^2 \CC[z]_d \cong \wedge^2\CC^{d+1} \cong \CC^{d(d+1)/2} \to \CC[z]_{2d-2} \cong \CC^{2d-1}
\end{equation}
given on the set $\wt{S}_d$ of decomposable $2$-vectors by
\begin{equation} \label{Wtilde-S}
\wt{W}(F_0 \wedge F_2) = F_0 F'_2 - F_2 F'_0\,.
\end{equation}
Now $\wt{W}(F_0,F_2) \equiv 0$ if and only if $F_0$ and $F_2$ are linearly dependent.   Hence $\wt{W}$ projectivizes to a map
\begin{equation} \label{W-S}
W:  S_d  \to \CP^{2d-2}\,,
\end{equation}
where
$S_d = \bigl\{[v \wedge w] \in P(\wedge^2 \CC^{d+1}): \ v,w \in \CC^{d+1} \text{ linearly independent}\!\bigr\}$, called the \emph{Wronski map}.
Note that $S_d$ can be thought of as the quotient of $\Hol_d(S^2,\CP^1)$ by the action of $\SL(2,\CC)$ given by \emph{post}composing $f = [F_0,F_2] \in \Hol_d(S^2,\CP^1)$ with a fractional linear transformation of the codomain:
\begin{equation} \label{action-post}
\left(\begin{matrix}F_0 \\ F_2 \end{matrix}\right) \mapsto
	\left(\begin{matrix}a F_0+b F_2 \\ c F_0+d F_2 \end{matrix}\right)
	\quad \text{for} \quad
	\left(\begin{matrix} a & b \\ c & d \end{matrix}\right) \in \SL(2,\CC).
\end{equation}
This action leaves $[F_0 \wedge F_2]$ invariant.
Further, as is well known, $S_d$ is the complex quadric
$\bigl\{ [\omega] \in \CP^{d(d+1)-1} : \omega \in \wedge^2\CC^{d+1}, \omega \wedge \omega = 0 \bigr\}$; this is biholomorphic to the Grassmannian $G_2(\CC^{d+1})$  via $[F_0 \wedge F_2] \mapsto \spn\{F_0, F_2\}$, and so has dimension $2d-2$.   The degree of the Wronski map is given by the $d$'th \emph{Catalan number}  $\frac{1}{d}\binom{2d-2}{d-1}$ see, for example, \cite{Gol}.

\begin{lem} \cite{Ej-min} The map $f = [F_0,F_2] \in \Hol_d(S^2, \CP^1)$ is a collapse point if and only if\/ $[F_0 \wedge F_2]$ is a critical point of the Wronski map \eqref{W-S}; equivalently, if and only if the kernel of the linear map \eqref{Wtilde} contains some vectors tangent to $\wt{S}_d$ at $F_0 \wedge F_2$\,.
\end{lem}

\begin{proof}
A simple calculation shows that the derivative of \eqref{Wtilde0} is given by
\begin{eqnarray*}
\dd\wt{W}_{(F_0,F_2)}(\U_3,-\U_1) &=&  - F_0 \U_1' + F_0' \U_1 - F_2 \U_3' + F_2' \U_3 \\
	&=& \{F_0, \U_1\} + \{F_2, \U_3\} \\
	&=& \dd Q_{(F_0,0,F_2,0)}(0,\U_1,0,\U_3) \,,
\end{eqnarray*}
where the second equality is from \eqref{Q-NF}, the bracket notation is as in \eqref{inf-hor}, and $Q$ is given by \eqref{Q}.
\end{proof}

Writing $F_0(z) = \sum_{i=0}^d a_i z^i$ and $F_2(z) = \sum_{i=0}^d c_i z^i$ for constants $a_i, c_i \in \CC$,  the map  \eqref{Wtilde0} is given by
$$
\wt{W}(F_0,F_2) = \sum_{i,j = 1}^d (j-i)a_i c_j z^{i+j-1} = \sum_{0 \leq i < j \leq d} (j-i)h_{ij} z^{i+j-1}
$$
where $h_{ij} = a_i c_j - a_j c_i$\,; thus, with respect to the basis $\{z^i \wedge z^j : 0 \leq i < j \leq d\}$ of $\wedge^2 \CC[z]_d$ and the basis $\{z^k: 0 \leq k \leq 2d-2\}$ of $\CC[z]_{2d-2}$\,,
the linear map \eqref{Wtilde}
is given by
$$
h_{ij} z^i \wedge z^j \mapsto (j-i)h_{ij} z^{i+j-1} \,.
$$
Hence \eqref{Wtilde} is represented by the $(2d-1) \times d(d+1)/2$ matrix $M_d$\,, with rows indexed by $k \in \{0,\ldots,2d-2\}$ and columns by $(i,j)$ where $0 \leq i < j \leq d$, with one non-zero entry in each column $(i,j)$, namely, $j-i$ at row $k = i+j-1$.

\subsection{Invariant Functions and formulae for small twistor degree} \label{subsec:invt}

As well as the action \eqref{action-post},  $\SL(2,\CC)$ acts on $\Hol_d(S^2,\CP^1)$ by \emph{pre}composition by a fractional linear transformation of the domain:
\begin{equation} \label{frac-linear}
z \mapsto \frac{az+b}{cz+d} \quad \text{for }z \in S^2 = \Ci \ , \ 
\left( \begin{matrix}a & b \\ c & d \end{matrix} \right) \in \SL(2,\CC).
\end{equation}
This gives an action on $\wedge^2\CC[z]_d$ and so on $S_d$.   It is clear that $C_d$ is invariant under this action, so that we expect it to be given as the zero set of polynomials invariant under this action.  We shall find such polynomials for  $d \leq 5$. This case is special because (i) $\Harm^{\ff}_d(S^2,S^4)$ is known to be a manifold for $d \leq 5$ \cite{BoWo-space}; (ii) from Remark \ref{rem:Kotani}, $\Harm^{\ff}_d(S^2,S^6)$ is empty, equivalently, the space of extra eigenfunctions has dimension $\leq2$ for all $f \in \Hol_d(S^2,\CP^1)$ if and only if $d \leq 5$.
We now study the cases $d=3, 4,5$ separately.

\medskip

\emph{The case $d=3$.}  In this case, the linear map \eqref{Wtilde} is given by
\begin{equation} \label{Wtilde-d3}
(h_{01},h_{02},h_{03},h_{12},h_{13},h_{23}) \mapsto (h_{01}, 2h_{02}, 3h_{03}+h_{12}, 2h_{13}, h_{23})
\end{equation}
and we have $\wt{S}_3 = \{  (h_{01},h_{02},h_{03},h_{12},h_{13},h_{23}) \in \CC^6 : r  = 0\}$
where $r = h_{01} h_{23} - h_{02} h_{13} + h_{03}h_{12}$\,.

The kernel of \eqref{Wtilde-d3} is clearly one-dimensional; it is tangent to $\wt{S}_3$ precisely when the $6 \times 6$ matrix  $\Mm_3$ made up of $M_3$ with the row vector $\grad r$ adjoined is singular.  Now
$$
\Mm_3 = \left( \begin{array}{rrrrrr} 1 & 0 & 0 & 0 & 0 & 0 \\
				0 & 2 & 0 & 0 & 0 & 0 \\
				0 & 0 & 3 & 1 & 0 & 0 \\
				0 & 0 & 0 & 0 & 2 & 0 \\
				0 & 0 & 0 & 0 & 0 & 1 \\
	h_{23} & -h_{13} & h_{12} & h_{03} & -h_{02} & h_{01} \end{array} \right)
$$
and this is singular precisely when $3 h_{03} - h_{12} = 0$.  Hence, the set $C_3 \in \Hol_3(S^2, \CP^1)$ of collapse points is the linear algebraic variety
$$
C_3 = \bigl\{[h_{01},h_{02},h_{03},h_{12},h_{13},h_{23}] \in S_d : 3 h_{03} - h_{12} = 0 \bigr\}\,.
$$

In more concrete terms, this says that $f = [F_0,F_2] \in C_3$ if and only if the coefficients of $F_0$ and $F_2$ satisfy $3a_0 c_3 - a_1 c_2 + a_2 c_1 - 3a_3 c_0 = 0$.

  It can be checked that $P_3 = 3h_{03} - h_{12}$ is indeed $\SL(2,\CC)$-invariant.  We interpret this polynomial. 
Guided by invariant theory (see, for example, \cite{Gur}), for any $d$, put $h_{ij} = \binom{d}{i}\binom{d}{j}H_{ij}$ \ $(i,j = 1, \ldots, d)$.
Then the fundamental invariant linear polynomial in the bivector $(H_{ij})$ is its \emph{trace}:  $\Tr \, H_{ij} = \sum_{i=1}^d(-1)^i \binom{d}{i} H_{i,n-i}$\,. This gives an invariant polynomial for any $d$, though it is zero when $d$ is even by the antisymmetry of $H_{ij}$\,.  To understand this better, write $\iI = (i_1, \ldots, i_d) \in I = \{0,1\}^d$.  Also write $i = |\iI| \equiv \sum_{t=1}^d i_t$\,, and similarly for other multi-indices.    Write $\ov{i_t}$ for the complement $1 - i_t$\,. Then, up to scale,
$$
P_3 = \Tr \, H_{ij} \equiv \sum_{\iI \in I }(-1)^i H_{ij}
$$
where $j = |\jJ|$ with $\jJ$ determined by the rule $j_t = \ov{i_t}$ \ $(t=1,\ldots, d)$.  It is easy to prove that this is the only invariant linear polynomial up to scale. 

\smallskip
For any $d$, we can, in principle, find polynomials giving $C_d$ in a similar way.   $\wt{S}_d$ is given by the vanishing of the $\binom{d+1}{4}$ equations
$$
r_{ijkl} \equiv h_{ij}h_{kl} - h_{ik}h_{jl} + h_{il}h_{jk} = 0 \qquad (0 \leq i < j < k < l \leq d) \,;
$$
the gradients of the $r_{ijkl}$ give a matrix of rank $(d-1)(d-2)/2 = \codim\, \wt{S}_d$\,.
We form a $\bigl(2d-1 + \binom{d}{4} \bigr) \times d(d+1)/2$ matrix $\Mm_d$ from $M_d$ by adjoining those gradients; then $C_d$ is the set on which $\Mm_d$ fails to have maximal rank.  We find this by calculating the highest common factor of all the minors of size $d(d+1)/2$, always working modulo the ideal $R$ generated by the $r_{ijkl}$, giving a polynomial $P_n$ in the $h_{ij}$.   As above, this is $\SL(2,\CC)$-invariant; we shall now identify it for $d=4,5$.

\medskip

\emph{The case $d=4$.}  Computer calculation of the above hcf gives
$$
P_4 = -2\,h_{{02}}h_{{24}}+3\,(h_{{03}}h_{{23}}+h_{{12}}h_{{14}})+18\,{h_{{04}}}^{\! 2}-9\,h_
{{04}}h_{{13}} -h_{{12}}h_{{23}}\, (\!\!\!\!\!\!\mod R).
$$
We shall express this in terms of invariant polynomials. Guided by the case $d=3$, we form the invariant polynomial
$$
P_4^1 = \sum_{\iI,\jJ \in I} (-1)^{i+j}H_{ik}H_{jl}
$$
where $k = |\kK|$ and $l = |\lL|$ with $\kK$ and $\lL$ determined by the rule $k_t = \ov{j_t}$, $l_t = \ov{i_t}$   \ $(t=1,\ldots, d)$.  We can define another invariant polynomial $P_4^2$ by the same formula but with the more complicated rule
$k_1 = \ov{i_1}$, $k_2 = \ov{i_2}$, $k_3 = \ov{j_1}$, $k_4 = \ov{j_2}$,
$l_1 = \ov{i_3}$, $l_2 = \ov{i_4}$, $l_3 = \ov{j_3}$, $l_4 = \ov{j_4}$.
Note that both these polynomials are sorts of traces.
Then computer studies show that $P_4^1$ and $P_4^2$ form a basis for the vector space of quadratic invariant polynomials modulo $R$, and that 
$$
P_4 = -\tfrac{1}{3} P_4^1 + \tfrac{4}{3} P_4^2 \quad  (\!\!\!\!\!\!\mod R). 
$$

\medskip

\emph{The case $d=5$.}  Similarly, we can define invariant polynomials $P_5^s$ \ $(s=1,2,3,4)$ by
$$
P_5^s = \sum_{\iI,\jJ,\rR \in I} (-1)^{i+j+r}H_{ik}H_{jl}H_{rs}
$$
where $k = |\kK|$, $l = |\lL|$ and $s = |\sS|$ with $\kK$, $\lL$ and $\sS$ determined by the following rules:

\smallskip

$(s=1)$: $k_t = \ov{i_t}$, $l_t = \ov{j_t}$, $s_t = \ov{r_t}$ $(t=1, \ldots, 5)$;

\smallskip

$(s=2)$: $k_t = \ov{i_t}$ $(t=1,2,3)$, $k_4=\ov{j_1}$, $k_5=\ov{j_2}$, $l_1=\ov{i_4}$, $l_2=\ov{i_5}$, $l_3=\ov{r_1}$, $l_4=\ov{r_2}$, $l_5=\ov{r_3}$, $s_1=\ov{j_3}$, $s_2=\ov{j_4}$, $s_3=\ov{j_5}$, $s_t=\ov{r_t}$ $(t=4,5)$;

\smallskip

$(s=3)$: $k_t = \ov{i_t}$ $(t=1,2,3,4)$, $k_5=\ov{j_1}$, $l_1=\ov{i_5}$, $l_2=\ov{r_1}$, $l_3=\ov{r_2}$, $l_4=\ov{r_3}$, $l_5=\ov{r_4}$, $s_1=\ov{j_2}$, $s_2=\ov{j_3}$, $s_3=\ov{j_4}$, $s_4=\ov{j_5}$, $s_5=\ov{r_5}$\,;

\smallskip

$(s=4)$: $k_1=\ov{r_4}$, $k_2=\ov{r_5}$, $k_3=\ov{i_1}$, $k_4=\ov{i_2}$, $k_5=\ov{i_3}$, $l_1=\ov{i_4}$, $l_2=\ov{i_5}$, $l_3=\ov{j_1}$, $l_4=\ov{j_2}$, $l_5=\ov{j_3}$, $s_1=\ov{j_4}$, $s_2=\ov{j_5}$, $s_3=\ov{r_1}$, $s_4=\ov{r_2}$, $s_5=\ov{r_3}$\,.

\medskip

Note that $P_5^1$ is the cube of the trace of $(H_{ij})$.  More invariant polynomials can be formed by similar rules.  However, computer studies show that these four form a basis for the invariant cubic polynomials modulo $R$, and that
$$
P_5 = \tfrac{671}{6} P_5^1 - \tfrac{125}{2} P_5^2 - 225 P_5^3 + \tfrac{125}{3} P_5^4
		\quad (\!\!\!\!\!\!\mod R).
$$

\subsection{The influence of branch points} \label{subsec:branch}

Let $R \in \CC_{2d-2}[z]$.  The inverse image of $R$ under the Wronski map \eqref{W-S} consists of elements $[F_0 \wedge F_2] \in S_d$ such that the position and order of the branch points of $f = [F_0,F_2]:S^2 \to \CP^1$ are given by the zeros of $R$.  In general, the number of inverse images = the degree of the Wronski map = the Catalan number $\frac{1}{d}\binom{2d-2}{d-1}$. When some of the inverse images coalesce, they give a critical point of $W$, i.e. a collapse point.

\smallskip

\emph{The case $d=3$}.  Then the Wronski map has degree $2$.  Any $f \in \Hol_3(S^2,\CP^1)$ has four branch points up to order.  
If the branch points of $f$ are not all simple, then, by a fractional linear transformation, they can be placed at, say, $0$, $1$ and $\infty$, then direct calculation of $P_3$ shows that $f$ is not a collapse point.  This also follows from Proposition \ref{prop:max-br-pt} below.

So we may suppose that $f$ has four distinct simple branch points.
By $SL(2,\CC)$-invariance, we can assume that three of these points are at $0$, $1$ and $\infty$; let $\alpha$ be the fourth one.
Then $R = z(z-1)(z-\alpha)$ and solving \eqref{Wtilde-d3} for $f \in S_d$, we find that
$$
[F_0,F_2] = [z^3+(6\be-2-2\al)z^2, -\tfrac{1}{2}z+\be]
$$
where $\be$ is a root of $12\be^2+(-4-4\al)\beta + \al = 0$\,.

These roots are usually distinct; they coalesce when the discriminant $\al^2- \al+1$ is zero giving $\al = (1 \pm i\sqrt{3})/2 = \eu^{\pm\ii\pi/3}$, the primitive cube roots of $-1$, so that $R = z^3+1$.  In general, a quadruplet of points in $S^2 = \Ci$ gives rise to six different values of the cross ratio.  It is called \emph{equianharmonic} if these values reduce to two, in which case the cross ratios are the above primitive cube roots of $-1$;  equivalently, by a fractional linear transformation, we can place the points at the vertices of a regular tetrahedron on $S^2$.   Taking those vertices to be $1,\omega,\omega^2,\infty$ where $\omega = \eu^{2\pi\ii/3}$ is a primitive cube root of $+1$, we obtain
\begin{equation} \label{can-d3}
[F_0,F_2] = [z^3+2, -\tfrac{1}{3}z]\,.
\end{equation}

Hence \emph{$f$ is a collapse point (equivalently, admits extra eigenfunctions) if and only if it is given by \eqref{can-d3} up to fractional linear transformations of the domain and codomain, equivalently, if and only if it has simple branch points which form an equianharmonic quadruplet}. This is proved by a different method in \cite{MoRo}.

\smallskip

\emph{The case $d\geq 4$}.  Then the Wronski map has degree at least $5$.  Any $f \in \Hol_d(S^2,\CP^1)$ has $2d-2 \geq 6$ branch points up to multiplicity and there are many possibilities for their positions and orders to consider; we give an example where the branch points exhibit symmetry through the origin.  We see that collapse points occur when the branch points are so configured that some preimages of the Wronski map coalesce.

\begin{exm}
Suppose that $d=4$ and that $f$ has six distinct simple branch points at $\{0,1,-1,\al,-\al, \infty\}$ Then $R = (z^2-1)(z^2 - \al)$.   Solving $\wt{W}(F_0,F_2) = R$ by computer gives the five solutions:
\begin{eqnarray*}
f_1 &= &[z^4 - \al^2,\, -z^2/2 + (\al^2+1)/4] \\
f_2 \text{ and } f_3 &= &\bigl[z^4+\be z^3 + \al\be z - \al^2,\, -\tfrac{1}{2}z^2 + \tfrac{1}{6}\be z - \tfrac{1}{6}\al \bigr] \\
f_4 \text{ and } f_5 &= &\bigl[z^4+\ga z^3 + \al\ga z - \al^2,\, -\tfrac{1}{2}z^2 + \tfrac{1}{6}\ga z + \tfrac{1}{6}\al \bigr]
\end{eqnarray*}
where $\be = \pm\sqrt{3\al^2+2\al+3}$ and $\ga = \pm\sqrt{3\al^2-2\al+3}$\,.

It can be checked that when $\be$ and $\ga$ are non-zero, none of the five solutions gives $P_4 = 0$, so are not collapse points.  When $\be = 0$, i.e., when $\al = (-1 \pm i\sqrt{8})/3$, the solutions $f_1$, $f_2$ and $f_3$ coalesce to give a collapse point.  Similarly, when $\ga = 0$, i.e., when $\al = (+1 \pm i\sqrt{8})/3$, the solutions $f_1$, $f_4$ and $f_5$ coalesce to give a collapse point.
\end{exm}

Note that the complexity of the problem grows quickly with the degree $d$.  Indeed, the Catalan number $\frac{1}{d}\binom{2d-2}{d-1}$, which gives the degree of the Wronski map, increases rapidly with $d$.  For $d=6$, it is $42$; for $d=9$, it is more than a thousand, for $d=20$ it is more than a billion.

We finish with a general result which shows that, if $f$ has a branch point of maximal order, then it is not a collapse point.

\begin{prop} \label{prop:max-br-pt}
Suppose that $f \in \Hol_d(S^2,\CP^1)$ has a branch point of order $r \geq d-1$ somewhere.  Then $r = d-1$ and $f$ is not a collapse point.
\end{prop}

\begin{proof}
We can assume that the branch point is at infinity so that $R$ has degree $2d-2-r \leq d-1$.  Now,  since $f = [F_0,F_2]$ has degree $d$, then one of $F_0$ or $F_2$ has degree exactly $d$.  Since $\deg R = \deg F_0 + \deg F_2 -1$ for $\deg F_2  < d$, we must have $r=d-1$, and $\deg F_2 = 0$, i.e. $F_2$ is a constant which we can take to be $1$.   But then
$\dd Q_{(F_0,F_2)}(0,0,0,\U_3) = \{1,\U_3\} = \U_3'$; on putting $\U_3$ equal to $z,\ldots, z^{d}$, this gives multiples of $1,\ldots, z^{d-1}$, respectively.
On the other hand, $\dd Q_{(F_0,F_2)}(0,\U_1,0,0) = \{F_0,\U_1\}$; on putting $\U_1$ equal to $z,\ldots, z^{d-1}$, this gives polynomials of degree exactly $d, \ldots, 2d-2$, respectively.   The image of $\dd Q_{(F_0,F_2)}$ thus contains a basis for $\CC[z]_{2d-2}$ so that $Q$ is submersive at $f$.  Hence $f$ is not a collapse point.
\end{proof}

Note that the first part of the result confirms the well-known result that the maximum possible order of a branch point of a map $f \in \Hol_d(S^2,\CP^1)$ is $d-1$.



\begin{thebibliography}{00}


\bibitem{AdSi} D. Adams and L. Simon,
Rates of asymptotic convergence near isolated singularities of geometric extrema,
{\it Indiana J. of Math.} {\volnostyle 37} (1988), 225--254.

\bibitem{AtHiSi} M.~F. Atiyah, N.~J.~Hitchin, and I.~M. Singer,
Self-duality in four-dim\-en\-sion\-al Riemannian geometry.
{\it Proc. Roy. Soc. London A}, {\volnostyle 362} (1978), 425--461.

\bibitem{BaWo-book} P.~Baird and J.~C.~Wood,
{\it Harmonic morphisms between Riemannian manifolds}
London Mathematical Society Monographs, No. 29,
Oxford University Press (2003). 

\bibitem{Ba} J.~L.~M.\ Barbosa,
On minimal immersions of $S\sp{2}$ in $S\sp{2m}$,
{\it Bull. Amer. Math. Soc.}, {\volnostyle 79} (1973), 110--114.

\bibitem{BoWo-higher} J. Bolton and L.~M. Woodward,
Higher singularities and the twistor fibration,
{\it Geom. Dedicata} {\volnostyle 80} (2000), 231--246.

\bibitem{BoWo-space}  J. Bolton and L.M. Woodward,
The space of harmonic two-spheres in the unit four-sphere,
{\it Tohoku Math. J. (2)}  {\bf 58}  (2006), 231--236. 

\bibitem{Br}  R.L. Bryant,
Conformal and minimal immersions of compact surfaces into the $4$-sphere,
{\it J. Differential Geom.}  {\volnostyle 17}  (1982), 455--473.

\bibitem{Ca1} E. Calabi,
Minimal immersions of surfaces in Euclidean spheres,
{\it J. Differential Geom.} {\volnostyle 1}(1967), 111--125.

\bibitem{Ca2} E. Calabi,
Quelques applications de l'analyse complexe aux surfaces d'aire minima,
in: {\it Topics in complex manifolds} (Univ. de Montr\'eal, 1967), 59--81.

\bibitem{Ch1} S.~S.\ Chern,
On the minimal immersions of the two-sphere in a space of constant curvature,
in {\it Problems in analysis (Lectures at the symposium in honor of Salomon
Bochner, Princeton University, Princeton, NJ., 1969)}, 27--40,
Princeton University Press (1970), Princeton, NJ.

\bibitem{Ch2} S.~S.\ Chern,
On minimal spheres in the four-sphere,
in: {\it Studies and essays (Presented to Yu-why Chen on his 60th birthday,
April 1, 1970)}, 137--150,
Math. Res. Center, Nat. Taiwan Univ., Taipei, 1970.

\bibitem{EeLe-Rpt1} J. Eells and L. Lemaire, 
A report on harmonic maps,
{\it Bull. London Math. Soc.}, {\volnostyle 10} (1978),  1--68.

\bibitem{EeLe-Sel} J. Eells and L. Lemaire,
{\it Selected topics in harmonic maps},
C.B.M.S. Regional Conf. Series {\volnostyle 50} (Amer. Math. Soc., 1983).

\bibitem{EeLe-Rpt2} J. Eells and L. Lemaire,
Another report on harmonic maps,
{\it Bull. London Math. Soc.} {\volnostyle 20} (1988), 385--524.

\bibitem{EeSal} J. Eells and S. Salamon,
Twistorial construction of harmonic maps of surfaces into four-manifolds,
{\it Ann. Scuola Norm. Sup. Pisa Cl. Sci.} (4) {\volnostyle 12}  (1985), 589--640 (1986).

\bibitem{EeSam} J. Eells and J.H. Sampson,
Harmonic mappings of Riemannian manifolds,
{\it Amer. J. Math.} {\volnostyle 86} (1964), 109--160.

\bibitem{EeWo} J. Eells and J.C. Wood,
Harmonic maps from surfaces to complex projective spaces,
{\it Adv. in Math.} {\volnostyle 49} (1983), 217--263.

\bibitem{Ej-min} N. Ejiri,
Minimal deformation of a nonfull minimal surface in $S\sp 4(1)$,
{\it Compositio Math.} {\volnostyle 90} (1994), no. 2, 183--209.

\bibitem{Ej-bound} N. Ejiri,
The boundary of the space of full harmonic maps of $S\sp 2$ into $S\sp {2m}(1)$ and extra eigenfunctions,
{\it Japan. J. Math. (N.S.)} {\volnostyle 24} (1998), 83--121.

\bibitem{EjKo-extra} N. Ejiri and M. Kotani,
Minimal surfaces in $S\sp {2m}(1)$ with extra eigenfunctions,
{\it Quart. J. Math. Oxford Ser.} (2) {\volnostyle 43} (1992), 421--440.

\bibitem{EjMi} N. Ejiri and M. Micallef,
Comparison between second variation of area and second variation of energy of a minimal surface,
{\it preprint}, arXiv:0708.2188 (2007).

\bibitem{Fer-Sn}  L. Fern\'andez,
On the space of harmonic $2$-spheres in the $m$-sphere,
{\it preprint}, available from {\tt http://199.219.158.116/$\sim$luisfernandez}.

\bibitem{Gol}  L.~R.  Goldberg,  Catalan numbers and branched coverings by the Riemann sphere.
{\it Adv. Math.} {\volnostyle  85}  (1991),  no. 2, 129--144.

\bibitem{GuOh} M.~A. Guest and Y. Ohnita,
Group actions and deformations for harmonic maps,
{\it J. Math. Soc. Japan} {\volnostyle 45} (1993), 671--704.

\bibitem{GuOsRo} R.~D. Gulliver, R.~Osserman and H.~L. Royden,
A theory of branched immersions of surfaces,
{\it Amer. J. Math.} {\volnostyle 95} (1973), 750--812.

\bibitem{GuWh} R. Gulliver and B. White, The rate of convergence
of a harmonic map at a singular point,
{\it Math. Ann.} {\volnostyle 283} (1989), 539--549.

\bibitem{Gur} G.~B. Gurevich, {\it Foundations of the theory of algebraic invariants},
Translated by J. R. M. Radok and A. J. M. Spencer, P. Noordhoff Ltd., Groningen  (1964).

\bibitem{KoNo} S. Kobayashi and K. Nomizu,
{\it Foundations of differential geometry,  Vol. 2},
(Interscience, New York, 1969),
(reprinted John Wiley \& Sons, Inc., New York 1996).

\bibitem{KoMa} J.~L. Koszul et B. Malgrange,
Sur certaines structures fibr\'ees complexes,
{\it Arch. Math.} {\volnostyle 9} (1958), 102--109.

\bibitem{Ko} M. Kotani, Harmonic $2$-spheres with $r$ pairs of extra eigenfunctions,
{\it Proc. Amer. Math. Soc.} {\volnostyle 125} (1997), 2083--2092.

\bibitem{LeWo1} L. Lemaire and J.~C. Wood,
On the space of harmonic $2$-spheres in ${\volnostyle C}P^2$,
{\it Internat. J. of Math.} {\volnostyle 7} (1996), 211--225.

\bibitem{LeWo2}  L. Lemaire and J.~C. Wood,
Jacobi fields along harmonic 2-spheres in ${\bf C}P^2$ are integrable,
J. London Math. Soc. (2)  {\volnostyle 66}  (2002), 468--486.

\bibitem{Lo-S4} B. Loo,
The space of harmonic maps of $S\sp 2$ into $S\sp 4$,
{\it Trans. Amer. Math. Soc.} {\volnostyle 313} (1989), 81--102.

\bibitem{MoRo}  S. Montiel and A. Ros,
Schr\"odinger operators associated to a holomorphic map,
in: {\it Global differential geometry and global analysis (Berlin, 1990)},  147--174, Lecture Notes in Math., 1481, Springer, Berlin (1991).

\bibitem{MoUr} S. Montiel and F. Urbano,
Second variation of superminimal surfaces into self-dual Einstein four-manifolds,
{\it Trans. Amer. Math. Soc.} {\volnostyle 349} (1997), 2253--2269.

\bibitem{O'N} B. O'Neill,
The fundamental equations of a submersion,
{\it Michigan Math. J.}, {\volnostyle 13} (1966), 459--469.

\bibitem{Sal} S. Salamon,
 Harmonic and holomorphic maps.
{\it Geometry seminar ``Luigi Bianchi" II---1984}, 
 161--224, Lecture Notes in Math., 1164, Springer, Berlin (1985).

\bibitem{Sim} B.~A. Sim\~oes,
Twistor constructions of harmonic morphisms and Jacobi fields,
{\it PhD thesis},
University of Leeds (2007).

\bibitem{Simon} L. Simon,
{\it Theorems on regularity and singularity of energy minimizing maps},
Lectures in Mathematics, ETH Z\"urich, Birkh\"auser (1996).

\bibitem{Ur} H. Urakawa,
{\it Calculus of variations and harmonic maps},
Translations of Mathematical Monographs, 132,
American Mathematical Society, Providence, RI (1993).

\bibitem{Ve2} J.-L. Verdier,
Applications harmoniques de $S\sp 2$ dans $S\sp 4$,
in: {\it Geometry today (Rome, 1984)},  267--282,
Progr. Math., {\volnostyle 60}, Birkh\"auser Boston, Boston, MA, 1985. 

\bibitem{Ve3} J.-L. Verdier, Applications harmoniques de $S\sp 2$ dans $S\sp 4$. II,
in: {\it Harmonic mappings, twistors, and $\sigma$-models (Luminy, 1986)},
World Sci. Publishing, Singapore (1988), 124--147.

\bibitem{Wo-Tokyo} J.~C. Wood,
Jacobi fields along harmonic maps,
in: {\it Differential Geometry and Integrable Systems
(Proceedings of the 9th MSJ-IRI, Tokyo 2000)},
Contemporary Mathematics, Amer. Math. Soc.  {\volnostyle 308}  (2002), 329--340.

\bibitem{Wo-Rome} J.~C. Wood,
Infinitesimal deformations of harmonic maps and morphisms, in: {\it International Journal of Geometric Methods in Modern Physics}, {\volnostyle 3}, Nos. 5-6 (2006), 933--956. Special Issue: Proceedings of the International Congress on Symmetry in Geometry and Physics in honour of Dmitri V. Alekseevsky (Rome, Italy, 14 to 17 September 2005).

\bibitem{Wo-Tenerife} J.~C. Wood
Infinitesimal deformations of harmonic maps, in: {\it Proceedings of the XV International Workshop on Geometry and Physics}, Puerto de la Cruz, Tenerife, Canary Islands, Spain, September 11--16, 2006,
Publ. de la RSME {\volnostyle 10} (2007), 105--115.


\end{thebibliography}
\end{document}